\documentclass[11pt]{article}
\usepackage[utf8]{inputenc}

\usepackage{amssymb,amsmath,amsfonts,amsthm}
\usepackage{latexsym}

\usepackage{hyperref}
\usepackage{color}

\usepackage{graphicx} 
\usepackage{ upgreek }
\usepackage{calrsfs}

\usepackage[a4paper, total={6.5in, 8.5in}]{geometry}

\theoremstyle{plain}
\newtheorem{thm}{Theorem}
\newtheorem{proposition}[thm]{Proposition}
\newtheorem{corollary}[thm]{Corollary}
\newtheorem{lemma}[thm]{Lemma}
\newtheorem{definition}[thm]{Definition}

\newtheorem{remark}[thm]{Remark}
\newtheorem{conjecture}[thm]{Conjecture}

\DeclareMathOperator{\grad}{grad}
\DeclareMathOperator{\trace}{trace}
\DeclareMathOperator{\Id}{id}
\DeclareMathOperator{\End}{End}
\DeclareMathOperator{\rank}{rank}

\DeclareMathOperator{\Hess}{Hess}
\DeclareMathOperator{\ddiv}{div}

\newcommand*{\R}{\mathbb{R}}

\newcommand*{\Hh}{\mathbb{H}}
\newcommand*{\Ss}{\mathbb{S}}

\newcommand*{\HH}{\mathcal{H}}

\title{Horospherical mean curvature functions and D'Atri spaces}
\author{Gerhard Knieper, JeongHyeong Park and Norbert Peyerimhoff}
\date{}

\begin{document}

\maketitle

\abstract{We consider simply connected Riemannian manifolds without conjugate points for which the horospherical mean curvature function is continuous, reversible and invariant under the geodesic flow. We show under mild additional %regularity and 
curvature tensor conditions that rank one manifolds in this family are automatically asymptotically harmonic. In particular, compact rank one manifolds of this kind must be locally symmetric spaces of negative curvature. Moreover, we show under the same conditions that %This implies, in particular, that 
rank one D'Atri spaces without conjugate points are harmonic.
An earlier result of this type was proved by Druetta for certain homogeneous D'Atri spaces.
% extending an earlier result by Druetta for a certain class of homogeneous D'Atri spaces.
% Earlier results of this type were proved by Heber and Druetta for certain homogeneous D'Atri spaces.
% extending an earlier result by Druetta for a certain class of homogeneous D'Atri spaces.
}

\tableofcontents

\section{Introduction and statement of results}
\label{sec:intro}

In this paper, we introduce and investigate a new class of Riemannian manifolds without conjugate points, which we call \emph{manifolds with invariant horospherical mean curvature functions}. Our main results are presented in Subsection \ref{subsec:intro1}. In Subsection \ref{subsec:intro2}, we briefly explain our motivation to study these manifolds.

\subsection{Results}
\label{subsec:intro1}

All our Riemannian manifolds $(M,g)$ are assumed to be connected and complete without mentioning. Simply connected Riemannian manifolds will be always denoted by $(X,g)$. If such a simply connected manifold has no conjugate points, the associated Busemann function $b_v \in C^1(X)$ a unit tangent vector $v \in SX$ is defined by
$$ b_v(x) = \lim_{t \to \infty} d(c_v(t),x)-t, $$
where $c_v: \R \to X$ is the geodesic with initial condition $\dot c_v(0) = v$ and $d: X \times X \to [0,\infty)$ is the distance function of $X$. If $(X,g)$ has $C^2$-Busemann functions, the \emph{horospherical mean curvature function} $h: SX \to \R$ is defined as
\begin{equation} \label{eq:hfirst} 
h(v) = \Delta b_v(\pi(v)), 
\end{equation}
where $\Delta = - \ddiv \circ \grad$ is the positive Laplacian and $\pi: SX \to X$ the footpoint projection. This notion stems from the fact that the level sets of Busemann functions are horospheres and $h(v)$ agrees with their mean curvature at $\pi(v)$. Since these concepts are invariant under isometries and each Riemannian manifold $(M,g)$ without conjugate points is given by $M = X/\Gamma$ with $\Gamma \subset {\rm{Iso}}(X)$ the group of deck transformations, the horospherical mean curvature function is also defined on the quotient $SM$, that is $h: SM \to \R$. In fact, the horospherical mean curvature function can also be defined via Jacobi tensors without the condition of $C^2$-Busemann functions which will be discussed in the next section (see Definition \ref{def:horomeancurvfunc}). 

In this paper, we use a generalization of the \emph{geometric rank} $\rank(v)$ of a unit vector $v \in SM$, due to Knieper \cite{K}, which was originally defined on Hadamard manifolds in \cite{BBE-85} as the dimension of parallel Jacobi fields along $c_v$. More details can be found in Section \ref{sec:background}. Note that this notion satisfies $\rank(v)=\rank(-v)$ (\emph{reversibility}) and $\rank(\phi^t v)=\rank(v)$ (\emph{invariance under the geodesic flow} $\phi^t$), and the rank of a manifold $\rank(M)$ is defined as the minimum $\min_{v \in SM} \rank(v)$.

At various places in our paper we refer to results provided in the source \cite{K0} and use consistently the following notation: Lemma I.2.4 refers to Lemma 2.4 on page 461 in Chapter 1 of \cite{K0} (and not to Lemma 2.4 in Chapter 2 of \cite{K0}).   

Of particular focus in the paper is the following class of manifolds.

\begin{definition} \label{def:manwithinvhoro} A Riemannian manifold $(M,g)$ is called a \emph{manifold with invariant horospherical mean curvature function} if $M$ does not have conjugate points and if its horospherical mean curvature function $h: SM \to \R$ has the following properties:
\begin{itemize}
    \item[(i)] $h$ is continuous.
    \item[(ii)] Reversibility: $h(v) = h(-v)$ for all $v \in SM$.
    \item[(iii)] Invariance under the geodesic flow: $h(\phi^t v) = h(v)$ for all $v \in SM$ and $t \in \R$.
\end{itemize}
\end{definition}

\begin{remark} In the case of a non-positively curved manifold $(M,g)$, the continuity condition (i) can be dropped from Definition \ref{def:manwithinvhoro}, since it follows already from the fact that those manifolds have continuous asymptote (see Remark \ref{rem:continasympt}).
\end{remark}

Special examples of manifolds with invariant horospherical mean curvature functions are \emph{asymptotically harmonic manifolds}, i.e., manifolds for which $h$ is constant. This special class emcompasses all Euclidean spaces, rank one symmetric spaces and Damek-Ricci spaces. Other examples of manifolds with invariant horospherical mean curvature functions, for which $h$ is not constant, are higher rank symmetric spaces of non-compact type and D'Atri spaces without conjugate points and continuous $h$.

Our first result states that the class of manifolds with invariant horospherical mean curvature functions is closed under taking Riemannian products. This result is proved in Section \ref{sec:background} and is closely related to an observation by Zimmer for asymptotically harmonic manifolds (\cite[Lemma 36]{Zi-12}).

\begin{thm} \label{thm:riemprod}
Let $(M_i,g_i)$, $i=1,2$, be two manifolds without conjugate points and $(M,g)$ be their Riemannian product.
Then $M$ is a manifold with invariant horospherical mean curvature function if and only if both $M_1, M_2$ are manifolds with invariant horospherical mean curvature functions.
\end{thm}

Next, we present the main results of this paper. 

\begin{thm} \label{thm:main1}
Let $(M,g)$ be a rank one manifold with invariant horospherical mean curvature function, whose curvature tensor together with its covariant derivative is uniformly bounded. Then $M$ is asymptotically harmonic.
\end{thm}

The proof of this rigidity result is given in Section \ref{sec:rankone}. Applying Theorem 1.5 in \cite{KP} yields the following consequence.

\begin{corollary} Let $(M,g)$ be a rank one manifold with invariant horospherical mean curvature function, whose curvature tensor together with its covariant derivative is uniformly bounded. Then $M$ has constant horospherical mean curvature $h$, is Gromov hyperbolic, its geodesic flow is Anosov, and its volume growth is purely exponential with volume entropy equals $h$.
\end{corollary}

Together with \cite[Theorem 3.6]{K}, which involves fundamental results by Benoist, Foulon, Labourie \cite{BFL,FL} and Besson, Courtois, Gallot \cite{BCG}, we have the following under the additional assumption of compactness.

\begin{corollary} Let $(M,g)$ be a compact rank one manifold with invariant horospherical mean curvature function, whose curvature tensor together with its covariant derivative is uniformly bounded. Then $M$ is a locally symmetric space of negative curvature.
\end{corollary}

Our second main result is concerned with D'Atri spaces without conjugate points. A D'Atri space is defined by the property that the volume density of geodesic spheres is invariant under reflection in their centers. Our result states that the rank one condition for D'Atri spaces with continuous horospherical mean curvature function implies harmonicity, which is defined by the stronger condition that the volume density of geodesic spheres is independent of their centers and only a function of the radius. 

\begin{thm} \label{thm:main2} Let $(M,g)$ be a rank one D'Atri space without conjugate points, whose curvature tensor together with its covariant derivative is uniformly bounded. If the horospherical mean curvature function $h: SM \to \R$ is continuous, then $M$ is harmonic and its geodesic flow is Anosov.
\end{thm}

This result extends an earlier result by Druetta for certain rank one \emph{homogeneous} D'Atri spaces without conjugate points and is proved in Section \ref{sec:rankoneDAtri}.

Note that Theorem 3.6 in \cite{K} implies again, for compact rank one D'Atri spaces without conjugate points and with continuous horospherical mean curvature functions are locally symmetric spaces of negative curvature.

\subsection{History, context and motivation}
\label{subsec:intro2}

Let us briefly explain our motivation to consider the new class of manifolds with invariant horospherical mean curvature functions.

% Let us briefly explain how this new class of manifolds with invariant horospherical mean curvature function emerged.

% In this paper, we introduce and investigate a new class of Riemannian manifolds without conjugate points. Let us briefly explain our motivation to study these manifolds.

We begin with the class of {\bf{ harmonic manifolds}}. They have a long history and emerged from the question whether the equation $\Delta f = 0$ on a Riemannian manifold has always a radial solution (see \cite[Paragraph 6.8]{Be}). This property is equivalent to the fact that the mean curvature of geodesic spheres depends only on the radius (\cite[Paragraph 6.19]{Be}). Examples of those manifolds, in the simply connected case, are all Euclidean spaces and symmetric spaces of negative curvature. The famous \emph{Lichnerowicz Conjecture} states that there are no other simply connected harmonic manifolds. Since harmonic manifolds are Einstein, this conjecture holds true in dimensions $2$ and $3$ (see \cite[Paragraph 6.52]{Be}). The $4$- and $5$-dimensional cases were confirmed by Walker \cite{W} in 1949 and Nikolayevsky \cite{Ni} in 2005, respectively.
A general proof of this conjecture for all simply connected compact manifolds was given in 1990 by Szabo \cite{Sz}. It is important to note that non-compact harmonic manifolds do not have conjugate points (see \cite[Th{\'e}or{\`e}me 2.1]{All}).
%since any noncompact manifold admits a geodesic ray without conjugate points and the volume homogeneity implies that there are no geodesic rays with conjugate points).
It came as a surprise, when Damek and Ricci found non-compact non-symmetric harmonic manifolds in 1992, starting in dimension $7$ (see \cite{DR}). Under the additional assumption of homogeneity, it was proved by Heber in \cite{He} that there are no further simply connected examples.
Even though there are partial classification results (see, e.g., \cite{RS,K}), a complete classification of all non-compact harmonic manifolds is still open.

Ledrappier \cite{Led1,LedShu} introduced {\bf{asymptotically harmonic manifolds}} (in the special context of simply connected manifolds without focal points) via various equivalent characterizations, one of them being manifolds without conjugate points whose horospheres have all the same constant mean curvature. Since horospheres are defined as the level sets of Busemann functions $b_v$ and can be viewed as limits of increasing geodesic spheres, this class contains all non-compact harmonic manifolds. It is open whether there exist asymptotically harmonic manifolds which are not harmonic.
% As a generalization of harmonic manifolds, Ledrappier \cite{Led1} introduced {\bf{asymptotically harmonic manifolds}} in 1990 via various equivalent characterizations, one of them being manifolds without conjugate points whose horospheres have the same constant mean curvature. Horospheres are defined as the level sets of Busemann functions $b_v$ and can be viewed as limits of increasing geodesic spheres. 
Asymptotically harmonic manifolds are not closed under taking Riemannian products, and higher rank symmetric spaces do also not belong to this class. This can be remedied by the weaker requirement that only equidistant horospheres have the same constant mean curvature. An equivalent condition, in the case of $C^2$ Busemann functions, is given by $\Delta b_v \equiv h(v)$ with a function $h$ on the unit tangent bundle, given by the constant mean curvature of the associated horosphere. We propose to refer to such spaces as \emph{weakly asymptotically harmonic spaces (wahm's)}. %In these spaces, the Laplacian of all Busemann functions $b_v$ is constant. 
In these spaces, all functions $e^{-h(v) b_v}$ are harmonic. 

Another widely studied class of Riemannian manifolds are {\bf{D'Atri spaces}} (see \cite{KPV}). They are defined by the property that the geodesic inversion is volume preserving or, equivalently, that geodesic inversion preserves the mean curvature of geodesic spheres. This class comprises (see \cite[Sections 4.1 and 4.4]{KPV}) both 
\begin{itemize}
\item \emph{naturally reductive homogeneous spaces} (see \cite[Theorem 2.3]{TV} for a geometric characterization of this class) and 
\item \emph{weakly symmetric spaces} (for each geodesic $c: \R \to X$ there is an isometry $f$ reversing the geodesic, $f(c(t))=c(-t)$, see \cite{Zi}). 
\end{itemize}
In order to include all D'Atri spaces without conjugate points in our new class of manifolds, we weaken the condition of the above-mentioned wahm's to the pointwise condition 
$$ h(v) = \Delta b_v(\pi(v)) = h(v) $$ 
for all unit tangent vectors $v$, that is, $\Delta b_v$ agrees with the mean curvature of the level set $b_v^{-1}(0)$ only at the point $\pi(v)$. An equivalent definition of these spaces is the condition that the horospherical mean curvature function $h(v)$ is invariant under the geodesic flow and reversible.    

For technical reasons, we assume additionally that the function $h$ is continuous, and we refer to manifolds with these very general invariance properties as \emph{manifolds with invariant horospherical mean curvature functions}. (In the special case of non-positive curvature or, more generally, without focal points, the manifolds have continuous asymptote and continuity of $h$ is automatically satisfied.) In this paper, we prove rigidity results for rank one manifolds in this class, where we use a slight generalization of the geometric rank notion in \cite{BBE-85} to manifolds without conjugate points, due to \cite{K}. For a survey on harmonic and asymptotically harmonic manifolds see also
\cite{K16}.

\section{Background, relevant concepts and notation}
\label{sec:background}

Of crucial importance in our investigations are stable and unstable Jacobi tensors. In Subsection \ref{subsec:jacobitensors}, we recall their fundamental properties and the definition of the geometric rank used in this paper. In Subsection \ref{subsec:riemprod}, we provide a proof that our class of manifolds with invariant horospherical mean curvature functions is closed under taking Riemannian products.

\subsection{Jacobi tensors, Riccati equation, Busemann functions and rank}
\label{subsec:jacobitensors}

The notation in this paper for Riemannian manifolds $(M,g)$ follows the notation used in \cite{KP}. The tangent bundle of $M$ is denoted by $TM$, the unit tangent bundle by $SM$, and the footpoint projection by $\pi: TM \to M$. For $v \in SM$, we denote by $c_v: \R \to M$ the corresponding geodesic with $\dot c_v(0) = v$. The geodesic flow of $M$ is denoted by $\phi^t: SM \to SM$. An orthogonal Jacobi tensor $J(t)\in \End(\phi^t(v))$ along $c_v$ is a solution of
$$ J''(t) + R_v(t) J(t) = 0, $$
where $J'' = \frac{D^2}{dt^2} J$ is the second covariant derivative of $J$ along $c_v$ and $R_v(t) \in \End_{\rm{sym}}(\phi^t(v))$ is the Jacobi operator along $c_v$, given by $R_v(t) = R(\cdot,\phi^t(v))\phi^t(v)$. 

From now on, we require additionally that our manifold has no conjugate points, to ensure that the following objects are well-defined. For $r > 0$, we consider the following orthogonal Jacobi tensors along $c_v$:
\begin{align*}
S_{v,r}(0) = \Id_{v^\bot}, &\quad S_{v,r}(r) = 0, \\
U_{v,r}(0) = \Id_{v^\bot}, &\quad U_{v,r}(-r) = 0.
\end{align*}
Note that we have $U_{v,r}(t) = S_{-v,r}(-t)$. 
The \emph{stable} and \emph{unstable Jacobi tensor} along $c_v$ are defined via the following initial conditions
$$ S_v(0) = \Id_{v^\bot}, \quad S_v'(0) = \lim_{r \to \infty} S_{v,r}'(0), $$
and similarly
$$ U_v(0) = \Id_{v^\bot}, \quad U_v'(0) = \lim_{r \to \infty} U_{v,r}'(0). $$
The existence of $\lim_{r \to \infty }S_{v,r}'(0)$ follows from the monotonicity $S'_{v,r}(0) <S'_{v,s}(0)$ for $r < s$ and
$S'_{v,r}(0) < U'_{v,1}(0)$ for all $r > 0$ (see proof of Lemma I.2.14 in \cite{K0} for the upper bound). Moreover, we introduce the symmetric endomorphisms
$$ S(v) = S'_v(0) \quad \text{and} \quad U(v) = U'_v(0). $$
Note that $U(v) = - S(-v)$ and we have the relation
\begin{equation} \label{eq:Srel} 
S_{\phi^{t_0} v}(t) = S_v(t+t_0) S_v^{-1}(t_0) 
\end{equation}
for all $t, t_0 \in \R$, since the Jacobi tensors on both sides have the same initial conditions at $t=0$.
To see that these endomorphisms are symmetric, we use the fact (see \cite[Lemma 1.2.2]{K})
\begin{equation} \label{eq:Wronski}
\Omega(Y_1,Y_2) = \rm{const} \quad \text{for all Jacobi tensors}, 
\end{equation}
where 
$$ \Omega(A,B)(t) = B^*(t) A'(t) - (B'(t))^* A(t), $$
is the Wronskian of $A(t),B(t) \in \End(\phi^t(v)^\bot)$ and $B^*(t)$ is the adjoint of $B(t)$, by choosing $Y_1=Y_2=S_{v,r}$ and using $S_{v,r}(r)=0$.

Since we also have $U(v) \ge S(v)$ the difference 
$$ D(v) = U(v) - S(v) $$ is a non-negative symmetric endomorphism of $v^\bot$. It follows from \eqref{eq:Srel} that
\begin{equation} \label{eq:S'rel} 
S(\phi^t v) = S'_v(t) S_v^{-1}(t), 
\end{equation}
and, similarly,
\begin{equation} \label{eq:U'rel}
U(\phi^t v) = U'_v(t) U_v^{-1}(t). 
\end{equation}
Both $S$ and $U$ satisfy the \emph{Riccati equation}, that is
\begin{equation} \label{eq:Ricatti} 
S'(\phi^t(v)) + S^2(\phi^t(v)) + R_v(t)S(\phi^t(v)) = 0, 
\end{equation}
In the special case of a simply connected Riemannian manifold $(X,g)$ without conjugate points and $C^2$-Busemann functions, 
$U(v)$ is the second fundamental form of the horosphere $\HH = b_v^{-1}(0)$ at $\pi(v)$, and $U(v)$ agrees with the Hessian of $b_v$ at $\pi(v)$, that is,
$$ \langle U(v)(w_1),w_2 \rangle = (\Hess b_v)(w_1,w_2) = \langle \nabla_{w_1} \grad b_v, w_2 \rangle \quad \text{for all $w_1,w_2 \in v^\bot$.} $$
For that reason, we call the following function $h$
the horospherical mean curvature function. 

\begin{definition} \label{def:horomeancurvfunc}
Let $(M,g)$ be a Riemannian manifold without conjugate points. The function $h: SM \to \R$, given by
$$ h(v) = \trace U(v), $$
is called the \emph{horospherical mean curvature function} of $M$.
\end{definition}

The next lemma shows that continuity of the horospherical mean curvature function allows to consider the Laplacian of Busemann functions.

% Note that Busemann functions of manifolds with continuous horospherical mean curvature functions are $C^2$. For the proof one shows that continuity of the horospherical mean curvature function of $(M,g)$ implies that $v \mapsto U(V)$ is continuous (see the proof of Lemma 17 in \cite{Zi-14}, which states this result in the special case of harmonic manifolds; however, the proof only uses the continuity of $v \mapsto \trace U(v)$).  

\begin{lemma} \label{lem:contasymp}
Let $(X,g)$ be a simply connected Riemannian manifold without conjugate points. If the horospherical mean curvature function $h: SX \to \R$ is continuous, then all Busemann functions $b_v$, $v \in SM$ and their associated horospheres, are $C^2$. Moreover, the map $v \to U(v)$ is continuous.
\end{lemma}

\begin{proof}
The proof is a combination of \cite[Theorem 1(i)]{Es} and the arguments given in the proof of Lemma 17 in \cite{Zi-14}. The proof in \cite{Zi-14} shows that continuity of $v \mapsto h(v) = \trace U(v)$  implies local uniform convergence (in $v \in SX$) of $U_{v,r} \to U(v)$ as $r \to \infty$. Therefore, Theorem 1(i) in \cite{Es} implies that all Busemann functions and associated horospheres are $C^2$. Moreover, the local uniform convergence and continuity of $v \to U_{v,r}$ for all $r > 0$ imply that $v \to U(v)$ is also continuous. 
\end{proof}

\begin{remark} \label{rem:continasympt} A simply connected Riemannian manifold $(X,g)$ without conjugate points for which the map $v \mapsto S(v)$ is continuous is called a \emph{manifold with continuous asymptote} (see \cite{Es}). In this case, the maps $v \mapsto U(v)$ and $h(v) = \trace U(v)$ are also continuous, since $U(v) = - S(-v)$. Manifolds with continuous asymptote contain the smaller class of \emph{manifolds with bounded asymptote}, that is, manifolds satisfying $\Vert S_v(t) \Vert \le C$ for all $v \in SX$ and $t \ge 0$ (see \cite[Satz 3.5]{KnDiss}). All manifolds of non-positive curvature or, more generally, without focal points, belong to the class of manifolds with bounded asymptote with constant $C=1$ (see \cite[Section 5]{Es}).
\end{remark}

Note that $h(v) = - \trace(S(-v))$ and, in the case $b_v \in C^2$, we have $h(v) = \Delta b_v(\pi(v))$. 
% and, more generally,
% $$ \Delta b_v(q) = h(-\grad b_v(q)). $$
This identity confirms that Definition \ref{def:horomeancurvfunc} of the horospherical mean curvature function generalizes our original definition \eqref{eq:hfirst} of the function $h$ in the introduction. Moreover, since $D(v) = U(v) - S(v) \ge 0$, we have
$$ h(v) + h(-v) = \trace D(v)  \ge 0. $$
This implied that $h \ge 0$ for manifolds with invariant horospherical mean curvature functions.

Next, we introduce the rank of a manifold without conjugate points.

\begin{definition} Let $(M,g)$ be a Riemannian manifold without conjugate points. The \emph{rank of a unit vector $v \in SM$} is defined as
$$ \rank(v) = 1 + \dim \ker D(v), $$
and the \emph{rank of $M$} is defined as
$$ \rank(M) = \min_{v \in SM} \rank(v). $$
\end{definition}

This definition of the rank for manifolds without conjugate point was introduced in \cite[Definition 3.1]{K}. It extends the earlier definition for Hadamard manifolds in \cite{BBE-85}, where $\rank(v)$ was defined as the dimension of parallel Jacobi fields along the geodesic $c_v$. The rank has the following invariance properties, which implies that there is a well-defined rank notion for geodesics, namely, the rank of any of its unit tangent vectors. 

\begin{proposition} Let $(M,g)$ be a Riemannian manifold without conjugate points. Then the rank
$$ \rank: SM \to \mathbb{N} $$
is invariant under the geodesic flow, that is,
$$ \rank(\phi^t v) = \rank(v) \quad \text{for all $v \in SM$ and all $t \in \R$,} $$
and \emph{reversible}, that is,
$$ \rank(v) = \rank(-v) \quad \text{for all $v \in SM$}.$$
\end{proposition}

\begin{proof}
  It follows from \eqref{eq:Wronski} that
  $$ D(v) = U(v) - S(v)^* = \Omega(U_v,S_v)(0) = \Omega(U_v,S_v)(t), $$
  and we have
  \begin{align*}
  \ker D(\phi^t v) &= \ker \left( U(\phi^t v)^* - S(\phi^t v)\right) \\
  &= 
  \ker \left( \left( U_v^{-1}(t) \right)^* \left[ (U_v'(t))^* S_v(t) - U_v(t)^* S_v'(t) \right] S_v^{-1}(t)\right) \quad \text{(by \eqref{eq:S'rel} and \eqref{eq:U'rel})}\\
  &= \ker \left( U_v(t)^* S_v'(t)-U_v'(t))^* S_v(t) \right) \\
  &= \ker \left( \Omega(U_v,S_v)(t)^* \right) = \ker D^*(v) = \ker D(v).
  \end{align*}
  This shows invariance of the rank under the geodesic flow. Reversibility follows from
  \begin{equation} \label{eq:Drev} 
  D(-v) = U(-v)-S(-v) = - S(v)-(-U(v)) = U(v)-S(v) = D(v). 
  \end{equation}
\end{proof}

The following proposition implies that an interesting classification of manifolds with invariant horospherical mean curvature functions starts with dimension $3$ (see Conjecture \ref{conj:3dimclass} at the end of this paper).

\begin{proposition} Let $(M,g)$ be a $2$-dimensional manifold with invariant horospherical mean curvature function. Then $M$ has non-positive constant curvature.
\end{proposition}

\begin{proof} The Riccati equation simplifies to the scalar identity along any unit speed geodesic $c: \R \to M$:
$$ u'(t) + u^2(t) + K(c(t)) = 0, $$
where $K(p)$ is the Gaussian curvature of $M$ at $p$. Since $u(t) = h(\dot c(t))$, invariance of $h$ under the geodesic flow implies
$$ -h^2(v) = K(\pi(v)) \quad \text{for all $v \in SM$.} $$
Therefore, $h$ is a function on $M$, and since any two points $p,q \in M$ can be connected by a geodesic, we conclude that $h$ is constant. Consequently, $M$ has also constant Gaussian curvature.
\end{proof}

% \begin{proposition} Let $(X,g)$ be a simply connected manifold with invariant horospherical mean curvature function. Then we have
% $$ \lambda_0(X) \ge \inf_{v \in SX} \frac{h(v)^2}{4}, $$
% where $\lambda_0(X)$ is the bottom of the spectrum of $\Delta$.
% \end{proposition}

\subsection{Riemannian products of manifolds with invariant horospherical mean curvature functions}
\label{subsec:riemprod}

In this subsection, we prove Theorem \ref{thm:riemprod} from the introduction. Our first result relates Busemann functions of Riemannian products with the Busemann functions of the factors.

\begin{proposition} \label{prop:riemprodbus}
Let $(X_i,g_i)$, $i=1,2$, be two simply connected Riemannian manifold and $(X=X_1 \times X_2,g)$ be their Riemannian product. Then $X$ has no conjugate points if and only if both $X_1$ and $X_2$ don't have conjugate points. In this case, we have the following relation between the Busemann functions for any unit vector $\tilde v = (\alpha v,\beta w) \in S_{(p,q)}X \subset T_p X_1 \oplus T_q X_2$ with $v \in S_p X_1$, $w \in S_q X_2$ and $\alpha^2+\beta^2=1$:
$$ b_{\tilde v}(x,y) = \alpha b_v(x) + \beta b_w(y). $$
In particular, the Busemann functions on $X$ are $C^2$ if and only if the Busemann functions on both factors are $C^2$, and we have
\begin{equation} \label{eq:buseprod} 
\Delta_X b_{\tilde v}(x,y) = \alpha \Delta_{X_1} b_v(x) + \beta \Delta_{X_2} b_w(y). 
\end{equation}
\end{proposition}

\begin{proof}
    Without loss of generality, we can assume $\alpha,\beta \ge 0$. Let $\tilde v = (\alpha v, \beta w)$ with $v \in S_p X_1$ and $w \in S_q X_2$. Then we have
    \begin{multline*}
    b_{\tilde v}(x,y) - \alpha b_v(x) -\beta b_w(y) = \lim_{t \to \infty} \left(d(c_{\tilde v}(t),(x,y)) -t \right) - \alpha b_v(x) -\beta b_w(y) \\ = \lim_{t \to \infty} \left(\sqrt{d^2(c_v(\alpha t),x) + d^2(c_w(\beta t),y)} -t \right) - \alpha \lim_{t \to \infty} \left( d(c_v(\alpha t),x) - \alpha t \right) - \beta \lim_{t \to \infty} \left( d(c_w(\beta t),y) - \beta t \right) \\
    = \lim_{t \to \infty} \left( \sqrt{d^2(c_v(\alpha t),x) + d^2(c_w(\beta t),y)} - \alpha d(c_v(\alpha t),x) - \beta d(c_w(\beta t),y) \right).
    \end{multline*}
    We can write
    \begin{align*}
    d(c_v(\alpha t),x) &= \alpha t + c_1(t), \\
    d(c_w(\beta t),y) &= \beta t + c_2(t)
    \end{align*}
    with functions $c_1, c_2: \mathbb{R} \to \mathbb{R}$ satisfying
    $$ | c_1(t) | \le d(x,p) \quad \text{and} \quad |c_2(t) | \le d(y,q). $$
    This leads to
    \begin{multline*}
    b_{\tilde v}(x,y) - \alpha b_v(x) -\beta b_w(y) = \lim_{t \to \infty} \left( \sqrt{(\alpha t + c_1(t))^2+(\beta t + c_2(t))^2} - \alpha(\alpha t + c_1(t)) - \beta (\beta t + c_2(t)) \right) \\
    = \lim_{t \to \infty} \left( t \sqrt{1+ \frac{2}{t} \left( \alpha c_1(t) + \beta c_2(t) \right) + \frac{1}{t^2}\left( c_1(t)^2 + c_2(t)^2 \right)} - t - \left(\alpha c_1(t) + \beta c_2(t) \right)\right) \\
    = \lim_{t \to \infty} \left( t \left( 1 + \frac{1}{t} \left( \alpha c_1(t) + \beta c_2(t) \right) + \frac{1}{2t^2} \left( c_1(t)^2 + c_2(t)^2 \right) + O\left( \frac{1}{t^2} \right) \right) - t - \left(\alpha c_1(t) + \beta c_2(t) \right) \right) \\ = \lim_{t \to \infty} \left( \frac{1}{2t} \left( c_1(t)^2 + c_2(t)^2 \right) + O\left( \frac{1}{t}\right) \right) = 0.
    \end{multline*}
    %The case $\alpha=0$ can be dealt with by an even simpler computation.
    The final relation follows from 
    $\Delta_X = \Delta_{X_1} \otimes \Id_{X_2} + \Id_{X_1} \otimes \Delta_{X_2}$.
\end{proof}

With this result at hand, we can now provide the proof of Theorem \ref{thm:riemprod}.

\begin{proof}[Proof of Theorem \ref{thm:riemprod}] Since all relevant concepts are invariant under isometries, we only need to prove the theorem for simply connected Riemannian manifolds $X = X_1 \times X_2$. The $C^2$-statement about the Busemann functions in Proposition \ref{prop:riemprodbus} implies that $h_X$ is continuous if and only if $h_{X_1}$ and $h_{X_2}$ are continuous. Moreover, \eqref{eq:buseprod} for unit vectors $\tilde v \in SX, v \in SX_1, w \in SX_2$ with $\tilde v = (\alpha v,\beta w)$ can be rewritten as 
$$ h_X(\tilde v) = \alpha h_{X_1}(v) + \beta h_{X_2}(w), $$
which show that $h_X$ is reversible and invariant under the geodesic flow if and only if both $h_{X_1}$ and $h_{X_2}$ have these properties, since the geodesic flows on these spaces are related as follows:
$$ \phi^t_X(\alpha v,\beta w) = (\alpha \phi^{\alpha t}_{X_1} (v), \beta \phi^{\beta t}_{X_2}(w)). $$
\end{proof}

% \begin{corollary} \label{cor:wahmprod}
% If $(M,g)$ is a non-trivial Riemannian product of two non-compact manifolds $(M_i,g_i)$, $i=1,2$, and $M$ is a weakly asymptotically harmonic manifold, then both $M_1$ and $M_2$ are also weakly asymptotically harmonic manifolds.
% \end{corollary}

% \begin{proof}
%     Let $M$ be weakly asymptotically harmonic with mean curvature function $h$ and $M=M_1 \times M_2$ and $v \in SM_1$. Fix $q \in M_2$. Choose $\tilde v =(v,0_q) \in SM$, and apply \eqref{eq:busemannprod} to obtain, for all $p \in M_1$,
%     $$ h(\tilde v) = \Delta_M b_{\tilde v}(p,q) = \Delta_{M_1} b_v(p). $$
%     Therefore $M_1$ is also weakly asymptotically harmonic with mean curvature function
%     $$ h_{M_1}(v) = h(v,0_q). $$
%     Similarly, $M_2$ is weakly asymptotically harmonic with mean curvature function
%     $$ h_{M_2}(w) = h(0_p,w) $$
%     for any fixed choice $p \in M_!$.
% \end{proof}

Similar considerations in the special setting of asymptotically harmonic manifolds were carried out in \cite{Zi-12}, based on stable Jacobi tensors instead of Busemann functions. In contrast to manifolds with invariant horospherical mean curvature functions, the class of asymptotically harmonic manifolds is generally not closed under taking Riemannian products. For an asymptotically harmonic manifold to be a non-trivial Riemannian product, its horospheres need to be minimal:

\begin{proposition}[see {\cite[Lemmas 36 and 37]{Zi-12}}] Let $M = M_1 \times M_2$ be a non-trivial Riemannian product of two manifolds without conjugate points. If $M$
is asymptotically harmonic, then both $M_1$ and $M_2$ are also asymptotically harmonic and all horospheres in $M,M_1,M_2$ are minimal horospheres. If, additionally, $M$ has no focal points, then $M$ is flat.
\end{proposition}

\section{Rank one manifolds with invariant horospherical mean curvature functions}
\label{sec:rankone}

This section is devoted to the proof of Theorem \ref{thm:main1} of the introduction. Since all concepts are invariant under isometries, it suffices to prove the theorem for simply connected rank one manifolds $(X,g)$ with invariant horospherical mean curvature functions. 

Recall that, for every $v \in SX$, $D(v) = U(v) - S(v)$ is a non-negative symmetric endomorphism on $v^\bot$ with $\trace D(v) =2h(v)$ and $\det D(v) \ge 0$. We consider, for $\alpha >0$, the open subset
$$ SX_\alpha = \{ v \in SX: \det D(v) > \alpha \}. $$
In the rank one case, this subset is non-empty and an open submanifold of $SX$ of the same dimension $n$, for small enough $\alpha >0$. Since (see \cite[Lemma I.2.17]{K0}),
\begin{equation} \label{eq:Dwbound} 
\Vert D(v) \Vert \le \Vert S'_v(0) \Vert + \Vert U'_v(0) \Vert \le 2 \sqrt{R_0} \quad \text{for all $v \in SX$}, 
\end{equation}
the largest eigenvalue of $D(v)$ is bounded above by $2 \sqrt{R_0}$. This implies that, for all $v \in SX_\alpha$, the smallest eigenvalue of the positive definite $D(\phi^t(v))$ is bounded below by
\begin{equation} \label{eq:alpharho} 
\rho = \alpha/(2 \sqrt{R_0})^{n-2}, 
\end{equation}
where $n = \dim(X)$.
In other words, we have
\begin{equation} \label{eq:Dlowbod} 
D(v) \ge \rho \cdot \Id \quad \text{for all $v \in SX_\alpha$,} 
\end{equation}
with $\alpha$ and $\rho$ related by \eqref{eq:alpharho}.

% In this section, we assume that $(M,g)$ is a rank one manifold with invariant horospherical mean curvature function. Moreover, we assume the existence of $R_0,R_0' >0$ such that $\Vert R \Vert \le R_0$ and $\Vert \nabla R \vert \le R_0'$. 

% Recall that, for every $v \in SM$, $D(v) = U(v) - S(v)$ is a non-negative symmetric endomorphism on $v^\top$ with $\trace D(v) =2h(v)$ and $\det D(v) > 0$ (the latter property follows from the rank one condition). As before, it suffices to prove this result for simply connected rank one manifolds $(X,g)$ with invariant horospherical mean curvature function.

The proof of Theorem \ref{thm:main1} for simply connected rank one manifolds $(X,g)$ with invariant horospherical mean curvature functions and curvature bounds $\Vert R \Vert \le R_0$ and $\Vert \nabla R \Vert \le R_0'$ for some $R_0, R_0' >0$
consists of the following four steps:
\begin{description}
\item[Step 1] $\det D$ and $\trace D$ are both reversible and invariant under the geodesic flow.
\item[Step 2] $\det D$ and $\trace D$ are both constant along weak stable and unstable manifolds of rank one vectors.
\item[Step 3] For all $\alpha >0$ the map $(p,v) \mapsto -\grad b_v(p)$ is continuous on the set $X \times SX_\alpha$.
\item[Step 4] $\det D$ and $\trace D$ are both globally constant.
\end{description}
The fourth step implies that $h=\frac{1}{2}\trace D$ is constant and $M$ is therefore asymptotically harmonic, finishing the proof. The individual steps of this proof are given in the following subsections.

%For the remainder of this section, $(X,g)$ denotes a simply connected manifold with invariant horospherical mean curvature function.

\begin{remark} For simply connected Riemannian manifolds $(X,g)$ with continuous asymptote, Eschenburg claims that the map $v \mapsto \grad b_v$ is continuous (see the proof of 
\cite[Theorem 1(ii)]{Es}). He argues that the vector field $\grad b_v$ is a solution of the following differential equation with continuous coefficients
$$ \nabla_w Y = U(Y(p)) w \quad \text{for all $w \in Y(p)^\bot$.} $$
While $\grad b_v$ is indeed a solution of this differential equation, we do not see how this yields the continuity of $v \mapsto \grad b_v$. Even the uniqueness of the solution is not clear to us, since a usual condition on the coefficients for such results is Lipschitz continuity. Therefore, we give an independent proof of this fact in our case in Step 3.
\end{remark}

\subsection{Step 1: $\det D$ and $\trace D$ are both reversible and invariant under the geodesic flow}
\label{subsec:detconstgeod}

In this section, we prove reversibility and invariance under the geodesic flow of $\det D$ and $\trace D$ for arbitrary manifolds $(M,g)$ with invariant horospherical mean curvature functions, independent of the rank and without any curvature bounds.

The reversibility of $\trace D$ and invariance of $\trace D$ under the geodesic flow follows readily from the corresponding properties of $h$. 

The reversibility of $\det D$ is a direct consequence of \eqref{eq:Drev}. Invariance under the geodesic flow is based on arguments which can be already found in the proof of Lemma 2.2 in \cite{HKSh-07} (see also proof of \cite[Proposition 3.2]{KP}). Introducing the operator
$$ H(v) = - \frac{1}{2} (S(v)+U(v)), $$
it follows from $HD+DH=S^2-U^2$ and the Riccati equation \eqref{eq:Ricatti} for the Jacobi tensors $S$ and $U$ that
$$ (HD+DH)(\phi^t v) = S(\phi^t v)^2 - U(\phi^t v)^2 = D'(\phi^t v). $$
We prove invariance of $\det D$ under the geodesic flow in the more general case without the rank one condition. 
In the case $\det D(\phi^t v) =0$ for all $t \in \R$, there is nothing to prove. In the case $\det D(\phi^{t_0} v) \neq 0$ for some $t_0 \in \R$, we conclude that
\begin{align*}
\det D(\phi^{t_0} v) \frac{d}{dt} \log \det D(\phi^{t_0} v) &= \trace \left[\left( \frac{d}{dt} D(\phi^{t_0} v\right) D^{-1}(\phi^{t_0} v) \right] \\
&= \trace\left[ H(\phi^{t_0} v) + (DHD^{-1})(\phi^{t_0} v) \right] = 2 \trace H(\phi^{t_0} v) \\
&= - \left( \trace S(\phi^{t_0}v) - \trace S(-\phi^{t_0}v)\right) \\
&= - \left(h(\phi^{t_0}v)-h(-\phi^{t_0}v)\right) = 0. 
\end{align*}
This implies that $\det D$ is constant along the geodesic flow.
\qed

\smallskip

The arguments in this subsection can also be used to prove the following result (cf. \cite[Corollary 2.6]{K} for a similar result in the context of harmonic manifolds).

\begin{proposition} Let $(M,g)$ be a manifold with invariant horospherical mean curvature function $h$. Then we have for all $v \in SX$,
\begin{equation} \label{eq:dettracerelation} 
\det D(v) \le \left( \frac{2h(v)}{n-1} \right)^{n-1}. 
\end{equation}
If \eqref{eq:dettracerelation} holds with equality for some $v \in SX$, we have for all $t \in \R$,
$$ R_{\phi^t(v)} = - \left( \frac{h(v)}{n-1} \right)^2 \Id_{\phi^t(v)^\bot}. $$
In particular, all sectional curvatures of planes $\Sigma$ containing $\phi^t(v)$ satisfy $K(\Sigma) = -\frac{h(v)^2}{(n-1)^2}$.
\end{proposition}

\begin{proof} Since $\trace D(v) = 2 h(v)$, \eqref{eq:dettracerelation} follows directly from the arithmetic-geometric mean inequality of the eigenvalues of $D(v)$. If this inequality holds with equality, all eigenvalues of $D(v)$ coincide and 
$$ D(v) = \frac{2h(v)}{n-1}  \Id_{v^\bot}. $$
Invariance of $h$ under the geodesic flow implies
$$ D(t) = D(\phi^t(v)) = \frac{2h(v)}{n-1} \Id_{\phi^t(v)^\bot}, $$
and therefore $D'(\phi^t(v)) = 0$. The above arguments show that $S^2(t) - U^2(t) = 0$. Since
$$ U(t) = S(t) + \frac{2h(v)}{n-1} \Id_{\phi^t(v)^\bot}, $$
we obtain
$$ U^2(t) = S^2(t) + \frac{4h(v)}{n-1} S(t) + \frac{4h(v)^2}{(n-1)^2} \Id_{\phi^t(v)}. $$
Since $U^2(t) = S^2(t)$, this leads to
$$ S(t) = - \frac{h(v)}{(n-1)} \Id_{\phi^t(v)} $$
and $S'(t) = 0$. Applying the Riccati equation, we conclude
$$ R_{\phi^t(v)} = - S^2(t) = - \frac{h(v)^2}{(n-1)^2} \Id_{\phi^t(v)}. $$
\end{proof}

\subsection{Step 2: $\det D$ and $\trace D$ are both constant along weak stable and unstable manifolds of rank one vectors}
\label{subsec:detconststable}

In this subsection, $(X,g)$ denotes a simply connected manifold with invariant horospherical mean curvature function $h: SX \to \R$ and satisfying $\Vert R \Vert \le R_0$ and $\Vert \nabla R \Vert \le R_0'$. Let $W^s(v)$ and $W^u(v)$ denote the stable and unstable leaf of $SX$ through $v \in SX$, respectively, that is,
\begin{align*}
W^s(v) &= \{ - \grad b_v(q): q \in b_v^{-1}(0) \}, \\
W^u(v) & = \{ \grad b_{-v}(q): q \in b_{-v}^{-1}(0) \}.
\end{align*}
We will also need weak stable and unstable leaves $W^{0s}(v), W^{0u}(v)$ through $v \in SX$, which are defined as follows:
\begin{align*}
W^{0s}(v) &= \{ - \grad b_v(q): q \in X \}, \\
W^{0u}(v) & = \{ \grad b_{-v}(q): q \in X \}.
\end{align*}
Our first aim is the proof of the following lemma.

\begin{lemma} \label{lem:localconst}
Assume $\rank(v) = 1$. Then there exists open neighbourhoods $V^s \subset W^s(v)$ and $V^u \subset W^u(v)$ of $v$ on which both $\det D$ and $\trace D$ are constant.
\end{lemma}

Let $\rank(v) = 1$. We first provide a proof of the lemma for the stable leaf $W^s(v)$. Since $\rank(v) = 1$, there exists $\alpha > 0$ such that $\det D(v) = 2 \alpha > 0$ and there exists an open neighbourhood $V^s \subset W^s(v)$ with $\det D(v') = \det D(\phi^t(v')) > \alpha$ for all $v' \in V^s$ and all $t \in \R$, due to the continuity of $D=U-S$ provided by Lemma \ref{lem:contasymp} and the invariance proved in the previous subsection. In other words, we have $V^s \subset SX_\alpha \cap W^s(v)$. 
Consequently, we have 
% Note that we have for all $w \in SX$ (see \cite[Lemma I.2.17]{K0}),
% \begin{equation} \label{eq:Dwbound} 
% \Vert D(w) \Vert \le \Vert S'_w(0) \Vert + \Vert U'_w(0) \Vert \le 2 \sqrt{R_0}, 
% \end{equation}
% that is, the largest eigenvalue of $D(w)$ is bounded above by $2 \sqrt{R_0}$. This implies that the smallest eigenvalue of the positive definite $D(\phi^t(v'))$ is bounded below by
% $\rho:= \rho'/(2 \sqrt{R_0})^{n-2} > 0$, where $n = \dim X$. In other words, we have
\begin{equation} \label{eq:Dcrucial} 
D(\phi^t(v)) \ge \rho \cdot \Id \quad \text{for all $v \in V^s$ and all $t \in \R$,} 
\end{equation}
with $\rho > 0$ related to $\alpha$ by \eqref{eq:alpharho}.
It suffices to prove 
\begin{equation} \label{eq:dettrace}
| \det D(v) - \det D(v')| = 0 \quad \text{and} \quad  | \trace D(v) - \trace D(v') | = 0
\end{equation}
for all $v' \in V^s$. 

For the proof of \eqref{eq:dettrace}, we need the following two propositions from \cite{KP}, which are both based on the crucial lower bound  \eqref{eq:Dcrucial} for the symmetric endomorphisms $D(v)$ fo $v \in \phi^\R(V^s)$. The first proposition is the following comparison result. It is stated in Section 3 of \cite{KP} under the assumption that $(X,g)$ is asymptotically harmonic, but its short proof goes also through in the more general setting of manifolds without conjugate points. 

\begin{proposition}[see {\cite[Lemma 3.3]{KP}}] 
\label{prop:Sdiffar}
There exists a constant $a \ge 1$, depending only on $R_0$ and $\rho$, such that
$$ 0 < S_{\phi^t(w)}'(0) - S_{\phi^t(w),r}'(0) \le \frac{a}{r} \quad \text{and} \quad 0 < U_{\phi^t(w),r}'(0) - U_{\phi^t(w)}'(0) \le \frac{a}{r} $$
for all $w \in V^s$, $r > 0$ and all $t \in \mathbb{R}$.
\end{proposition}

Let $v' \in V^s$ and $\gamma: [0,1] \to V^s \subset W^s(v)$ be a continuously differentiable curve with $\gamma(0) = v$ and $\gamma(1) = v'$. Let $\gamma_t = \phi^t \circ \gamma$. In the second proposition, we compare symmetric endomorphisms on different vector spaces $\gamma_t(0)^\bot$ and $\gamma_t(1)^\bot$. For any fixed $t \in \R$, we introduce an orthonormal frame $E_1^t,\dots,E_{n-1}^t: [0,1] \to TM$ of $\gamma_t^\bot$ along the footpoint curve $\beta_t = \pi \circ \gamma_t: [0,1] \to \mathcal{H}(t)$ in the horosphere $\mathcal{H}(t)$ corresponding to $W^s(\phi^t(v))$ as follows: $E_1^t(0),\dots,E_{n-1}^t(0),\phi^t(v)$ is an orthonormal basis of $T_{\pi(\phi^t(v))}X$ and $E_j^t$ is the parallel extension of $E_j^t(0)$ inside the hypersurface $\mathcal{H}(t) \subset X$ along $\beta_t$ for $j=1,\dots,n-1$. A symmetric endomorphism $A \in {\rm{End}}_{\rm{sym}}(\gamma_t(s)^\bot)$ can then be identified with the symmetric matrix with entries $\langle A (E_j^t(s)), E_k^t(s) \rangle_{1 \le j,k \le n-1}$.

% \bigskip
% \bigskip

% The proof that $\det D$ and $\trace D$ is constant along stable and unstable manifolds of rank one vectors uses the arguments given in Sections 2 and 3 of \cite{KP}.
% For the reader's convenience, we present the main ingredients and arguments of this proof. The first ingredient is the following fact, which holds for all manifolds without conjugate points.

\begin{proposition}[see {\cite[Corollary 2.6 and Theorem 2.8]{KP}}] \label{prop:SLbeta-estimate}
Let $\beta_t = \pi(\gamma_t)$ for all $t \in \mathbb{R}$. Then there exists a function $b: \mathbb{R} \to (0,\infty)$, depending only on $R_0, \rho$ and $n$, such that we have 
\begin{equation} \label{eq:btlength} 
\Vert \dot \beta_t(s) \Vert \le b(t) \Vert \dot \beta_0(s) \Vert \quad \text{for all $s \in [0,1]$ and $t \in \mathbb{R}$.} 
\end{equation}
Moreover, we have for $t \ge 0$
\begin{equation} \label{eq:btfunc} 
b(t) \le a e^{-\frac{\rho}{2} t}, 
\end{equation}
with a constant $a$ depending only on $R_0, \rho$ and $n$.

Let $r > 1 $. Then there exists a constant $C > 0$, depending only on $R_0,R_0',\rho,r$ and the dimension $n$, such that, for all $t \in \R$,
\begin{equation} \label{eq:Svtv} 
\Vert S_{\phi^t(v'),r}'(0) - S_{\phi^t(v),r}'(0) \Vert \le C \cdot \ell(\beta_t) 
\end{equation}
and 
\begin{equation}
\label{eq:Uvtv}
\Vert U_{\phi^t(v'),r}'(0) - U_{\phi^t(v),r}'(0) \Vert \le C \cdot \ell(\beta_t), 
\end{equation}
where $\ell(\beta_t)$ is the length of the curve $\beta_t$,
$S_{\phi^t(v),r}'(0), U_{\phi^t(v),r}'(0)$ and $S_{\phi^t(v'),r}'(0), U_{\phi^t(v'),r}'(0)$ are represented by their matrices with respect to the basis $E_1^t(0),\dots,E_{n-1}^t(0)$ and $E_1^t(1),\dots,E_{n-1}^t(1)$, respectively, and the norms chosen in \eqref{eq:Svtv} and \eqref{eq:Uvtv} are the operator norms for symmetric $(n-1)\times(n-1)$ matrices. 
\end{proposition}

% Note that property \eqref{eq:Dcrucial} is of crucial importance for these two propositions and that
Corollary 2.6 and Theorem 2.8 in \cite{KP} states the result for smooth curves $\gamma$, but it can be checked that continuous differentiability of $\gamma$ is sufficient. Having recalled these two results, we can now continue with our proof of \eqref{eq:dettrace}.

To prove \eqref{eq:dettrace}, we observe that
\begin{align*} 
| \trace D(v) - \trace D(v') | &= | \trace D(\phi^t(v)) - \trace D(\phi^t(v')) | \\
&\le (n-1) \Vert D(\phi^t(v)) - D(\phi^t(v')) \Vert \\ &\le (n-1)\left( \Vert S_{\phi^t(v)}'(0) - S_{\phi_t(v')}'(0) \Vert + \Vert U_{\phi^t(v)}'(0) - U_{\phi^t(v')}'(0) \Vert \right). 
\end{align*}
Similarly, using that $D(v), D(v')$ are matrices in a compact set by \eqref{eq:Dwbound} and the fact the determinant is a smooth function, we have a uniform Lipschitz constant $A > 0$, only depending on $R_0$ and the dimension $n$, such that, for all $t \ge 0$,
\begin{align*} 
| \det D(v) - \det D(v') | &= | \det D(\phi^t(v)) - \det D(\phi^t(v')) | \\
&\le A \Vert D(\phi^t(v)) - D(\phi^t(v')) \Vert \\ &\le A\left( \Vert S_{\phi^t(v)}'(0) - S_{\phi_t(v')}'(0) \Vert + \Vert U_{\phi^t(v)}'(0) - U_{\phi^t(v')}'(0) \Vert \right). 
\end{align*}
% It suffices to prove $| \det D(v) - \det D(\tilde v)| = 0$. 
% Using again \eqref{eq:Dwbound} and the fact that
%Since we have for all $w \in SX$ (see \cite[Lemma I.2.17]{K0}),
%$$ \Vert D(w) \Vert \le \Vert S'_w(0) \Vert + \Vert U'_w(0) \Vert \le 2 \sqrt{R_0}, $$
%and since 
% the determinant is a smooth function, we have a uniform Lipschitz constant $A > 0$, only depending on $R_0$ and the dimension $n$, such that, for all $t \ge 0$,
% \begin{align*} 
% | \det D(v) - \det D(\tilde v) | &= | \det D(\phi^t v) - \det D(\phi^t \tilde v) | \\
% &\le A \Vert D(\phi^t v) - D(\phi^t \tilde v) \Vert \\ &\le A\left( \Vert S_{\phi^t v}'(0) - S_{\phi_t \tilde v}'(0) \Vert + \Vert U_{\phi^t v}'(0) - U_{\phi^t \tilde v}'(0) \Vert \right). 
% \end{align*}
Therefore, it suffices to prove, for every $\delta > 0$, that there exists $t \ge 0$ with 
\begin{equation} \label{eq:SUdelta} 
\Vert S_{\phi^t(v)}'(0) - S_{\phi^t(v')}'(0) \Vert < \delta \quad \text{and} \quad \Vert U_{\phi^t(v)}'(0) - U_{\phi^t (v')}'(0) \Vert < \delta. 
\end{equation}
We apply Proposition \ref{prop:Sdiffar} and choose $r > 0$ large enough to have
\begin{equation} \label{eq:SUv'v'r} 
\Vert S_{\phi^t(w)}'(0) - S_{\phi^t(w),r}'(0) \Vert < \frac{\delta}{3} \quad \text{and} \quad \Vert U_{\phi^t(w)}'(0) - U_{\phi^t(w),r}'(0) \Vert < \frac{\delta}{3} \quad \text{for all $w \in V^s$ and all $t$}. 
\end{equation}
Proposition \ref{prop:SLbeta-estimate} implies that we have, for all $t \ge 0$,
\begin{align*} 
\Vert S_{\phi^t(v'),r}'(0) - S_{\phi^t(v),r}'(0) 
\Vert, \Vert U_{\phi^t(v'),r}'(0) - U_{\phi^t(v),r}'(0) 
\Vert &\stackrel{\eqref{eq:Svtv},\eqref{eq:Uvtv}}{\le} C  \ell(\beta_t) \\ &\stackrel{\eqref{eq:btlength}}{\le} b(t) \cdot C \ell(\beta) \\
&\stackrel{\eqref{eq:btfunc}}{\le} e^{-\frac{\rho}{2}t}\cdot a C \ell(\beta).
\end{align*}
Choosing $t \ge 0$ large enough, we have
$$ \Vert S_{\phi^t(v'),r}'(0) - S_{\phi^t(v),r}'(0) 
\Vert, \Vert U_{\phi^t(v'),r}'(0) - U_{\phi^t(v),r}'(0) 
\Vert < \frac{\delta}{3}, $$
and combining this with \eqref{eq:SUv'v'r}, we obtain \eqref{eq:SUdelta}. This finishes the proof of the lemma for the stable leaf $W^s(v)$. 

To prove the analogous result for the unstable leaf, we use the relation $W^u(v) = - W^s(-v)$. Then the above proof shows that there exists an open neighbourhood $V^s \subset W^s(-v)$ of $-v$ on which $\det D$ and $\trace D$ are both constant. Using reversibility of $\det D$ and $\trace D$, we conclude that both functions are constant on $V^u = - V^s \subset W^u(v)$. This finishes the proof of Lemma \ref{lem:localconst}. \qed

\smallskip

Our above local result yields the following global result.

\begin{corollary} Assume  $\rank(v) = 1$. Then both $\det D$ and $\trace D$ are constant on $W^{0s}(v)$ and $W^{0u}(v)$.
\end{corollary}

\begin{proof}
    By Subsection \ref{subsec:detconstgeod}, it suffices to prove the corollary for the stable and unstable manifolds $W^s(v)$ and $W^u(v)$. Since $\rank(v) = 1$, we have $\delta = \det D(\tilde v) > 0$. We consider the closed, non-empty subset
    $$ W = \{ w \in W^s(v): \det D(w) = \delta \} $$
    of $W^s(v)$. Since $W$ is also open in $W^s(v)$, by Lemma \ref{lem:localconst} and $W^s(v)$ is connected, we conclude that $W = W^s(v)$, and $\det D$ is constant on $W^s(v)$. In particular, all vectors in $W^s(v)$ are rank one. The same arguments can now be used to show that $\trace D$ is also constant on $W^s(v)$. The proof for $W^u(v)$ is analogous.
\end{proof}

\subsection{Step 3: Continuity of $(p,v) \mapsto - \grad b_v(p)$ on the set $X \times SX_\alpha$}

In this subsection, we assume that $(X,g)$ is a rank one simply connected manifold with invariant horospherical mean curvature function, satisfying $\Vert R \Vert \le R_0$ and $\Vert \nabla R \Vert \le R_0'$. It follows from the previous subsections that, for $\alpha >0$,  
% In this subsection, we assume that $(X,g)$ is a simply connected rank one manifold with invariant horospherical mean curvature. We define
% $$ SX_\alpha = \{ v \in SX: \det D(v) > \alpha \}, $$
% which is a non-empty open submanifold of $SX$ for $\alpha >0$ small enough, which is invariant under the geodesic flow. 
% Then we have
% $$ D(v) > \rho \Id \quad \text{for all $v \in SX_\alpha$} $$
% with $\rho = \frac{\alpha}{\sqrt{2 R_0}^{n-2}}$.
% For $v \in SX_\alpha$ and $q \in X$, we also have $-\grad b_v(q) \in SX_\alpha$. This implies that
$W^{0s}(v), W^{0u}(v) \subset SX_\alpha$ for any vector $v \in SX_\alpha$. This subsection is devoted to the proof that the map $(p,v) \mapsto -\grad b_v(p)$ is continuous on $X \times SX_\alpha$ for any fixed $\alpha >0$. Our proof is based on the following results.

\begin{proposition}[Divergence of geodesic rays on the set $SX_\alpha$]
  \label{prop:divgeod}
  Let $\alpha >0$ and $v,w \in SX_\alpha$ be two different unit vectors with $p = \pi(v) = \pi(w)$. Then we have
  \begin{equation} \label{eq:dcvcw} 
  d(c_v(t),c_w(t)) \to \infty \quad \text{as $t \to \infty$.} 
  \end{equation}
\end{proposition}

\begin{proof} Let $v,w \in SX_\alpha$ be two different unit vectors with the same footpoint $p$.
    For $\epsilon >0$, let 
    $$ C_\epsilon(v) = \left\{ v' \in S_pX: \sphericalangle_p(v,v') < \epsilon \right\} $$
    be a cone around $v$. Since $SX_\alpha \subset SX$ is open, we can find $\epsilon > 0$ small enough such that $C_\epsilon(v) \subset SX_\alpha$ and $w \not\in C_\epsilon(v)$. Let $t > 0$ be arbitrary and $\beta: [0,1] \to X$ be a curve with $\beta(0) = c_v(t)$ and $\beta(1) = c_w(t)$. If there exists $s \in (0,1)$ with $\beta(s) \in B_{t/2}(p)$, then $\ell(\beta) \ge t$. If $\beta$ satisfies $\beta([0,1]) \subset X \setminus B_{t/2}(p)$, then we can write
    $$ \beta(s) = {\rm{exp}}_p(r(s)v(s)) $$
    with $r: [0,1] \to [t/2,\infty)$ and $v: [0,1] \to S_pX$, $v(0)=v$, $v(1)=w$. Since $w \not\in C_\epsilon(v)$, we can find $s_0 \in (0,1)$ such that $v([0,s_0)) \subset C_\epsilon(v) \setminus B_{t/2}(p) \subset SX_\alpha$ and $v(s_0) \in \partial C_\epsilon(v)$, and we have
    $$ \ell(\beta) \ge \int_0^{s_0} \Vert \dot \beta(s) \Vert ds \ge \int_0^{s_0} \Vert A_{v(s)}(r(s))(P^{v(s)}_{r(s)}v'(s)) \Vert ds \ge c_1 e^{\frac{\rho t} {2}} \sphericalangle_p(v,v(s_0)) = c_1 e^{\frac{\rho t}{2}} \epsilon $$
with a suitable constant $c_1$, only depending the lower sectional curvature bound, by \cite[p. 113]{Bol-79}. Here $A(v)$ denotes the orthogonal Jacobi tensor along $c_v$ satisfying $A(0) = 0$ and $A'(0) = \Id$ and $P^{v}_{r}$ denotes parallel transport along $c_v$ from $c_v(0)$ to $c_v(r)$. Since $\beta$ was an arbitrary curve connecting $c_v(t)$ and $c_w(t)$, this implies that we have \eqref{eq:dcvcw}.
\end{proof}

A natural equivalence relation on the set of geodesic rays $c_v: [0,\infty) \to X$ is given by the condition that they stay in bounded distance, that is $c_v \sim c_w$ if 
$$ \sup_{t \ge 0} d(c_v(t),c_w(t)) < \infty. $$
In this case, the geodesic rays $c_v, c_w$ are called \emph{asymptotic}. The following proposition states that, geodesic rays with initial vectors in the same weak stable leaf in $SX_\alpha$ are asymptotic.

In the next proposition, $d_{\mathcal{H}}: \mathcal{H} \times \mathcal{H} \to [0,\infty)$ denotes the intrinsic distance in a horosphere $\mathcal{H} \subset X$ and $P_{\mathcal{H}}: X \to \mathcal{H}$ denotes the orthogonal projection of $X$ onto $\mathcal{H}$.

\begin{proposition}[Asymptotic geodesics on the set $SX_\alpha$] \label{prop:asgeod}
  Let $\alpha > 0$, $v \in SX_\alpha$, $w \in W^{0s}(v) \subset SX_\alpha$ and $\mathcal{H} = \pi(W^s(v)) \subset X$. Then we have
  $$ d(c_v(t),c_w(t)) \le d(\pi(v),\pi(w)) + a_2 e^{- \frac{\rho}{2} t} d_{\mathcal{H}}(\pi(v),P_{\mathcal{H}}(\pi(w))) \quad \text{for all $t \ge 0$,} $$
  where $\rho >0$ is given by \eqref{eq:alpharho} and $a_2 > 0$ is a constant only depending on $R_0$ and $\rho$. 
  % {\color{blue}curvature bound $R_0$ and and $D$ bound  $\rho$ ? \\ 
% Since, by your JDG paper Proposition 2.5(b),  if we assume $\|R_v(t)\|\le R_0$ for all $t\in\mathbb{R}$, and if, in addition,  $D\phi_t(v)\ge \rho\,\mathrm{id}$ for all $t\in\mathbb{R}$ and some $\rho >0$, then there exists
 % $a_2=a_2(R_0,\rho)>0$ such that, for all $r>1$ and $0\le t<r$,then $$\|S_{v,r}(t)\|\;\le\; a_2\,e^{-\frac{\rho}{2}\,t}, \quad
% \|U_{v,r}(-t)\|\;\le\; a_2\,e^{-\frac{\rho}{2}\,t}.$$}
\end{proposition}

\begin{proof} Let $v \in SX_\alpha$ and $t_0 \in \mathbb{R}$ the unique value such that $w_0 = \phi^{t_0} w \in W^s(v)$. Note that $d(\pi(w),\mathcal{H}) = |t_0| = |b_v(\pi(w))| \le d(\pi(v),\pi(w))$. Let $\gamma: [0,1] \to W^{s}(v)$ be a $C^1$-curve satisfying $\gamma(0)=v$ and $\gamma(1)=w_0$. Let $F(s,t) = c_{\gamma(s)}(t)$ be the corresponding geodesic variation and $\beta$ and $\beta_t$ be the footpoint curve of $\gamma$ and $\phi^t \gamma$, respectively. Then we have
$$
\ell(\beta_t) = \int_0^1 \Vert \dot \beta_t(s) \Vert ds = \int_0^1 \Vert S_{\gamma(s)}(t) (P^{\gamma(s)}_t \dot \beta(s)) \Vert ds \stackrel{(*)}{\le} a_2 e^{-\frac{\rho}{2} t} \int_0^1 \Vert \dot \beta(s) \Vert ds = a_2 e^{-\frac{\rho}{2} t} \ell(\beta),
$$
where $P^{\gamma(s)}_t$ denotes parallel transport along $c_{\gamma(s)}$ and where we used \cite[Proposition 2.5(b)]{KP} in $(*)$. Since $\gamma$ was an arbitrary curve connecting $v$ and $w_0$ inside $\mathcal{H}$ and $\beta_t(0)= c_v(t)$ and $\beta_t(1) = c_{w_0}(t) = c_w(t+t_0)$, we have
\begin{align*} 
d(c_v(t),c_w(t)) \le d(c_v(t),c_{w_0}(t)) + d(c_{w_0}(t),c_w(t))  \\
\le a_2 e^{-\frac{\rho}{2} t}d_{\mathcal{H}}(\pi(v),\pi(w_0)) + |t_0|.
\end{align*}
To finish the proof, observe that we have $\pi(w_0) = P_{\mathcal{H}}(\pi(w))$.
\end{proof}

A direct consequence of Propositions \ref{prop:divgeod} and \ref{prop:asgeod} is the following: For every $v \in SX_\alpha$ and every $p \in X$, there exists precisely one vector $w \in SX_\alpha$ with $\pi(w) = p$, such that the geodesic rays $c_v, c_w: [0,\infty) \to X$ stay in bounded distance, and this vector is $-{\rm{\grad}}\, b_v(p)$. Since geodesic rays are in $1-1$ correspondence with their start vectors, the equivalence relation of asymptotic geodesic rays in $SX_\alpha$ corresponds to the following equivalence relation on $SX_\alpha$: $v, w \in SX_\alpha$ are equivalent if $w = - {\rm{grad}}\, b_v(\pi(w))$.

To prove continuity of the map $(p,v) \mapsto - {\rm{grad}}\, b_v(p)$, we need one more result for arbitrary simply connected Riemannian manifolds without conjugate points and with a global lower curvature bound, which is stated in the following lemma.

\begin{lemma} Let $(X,g)$ be a simply connected 
Riemannian manifold without conjugate points and lower curvature bound $-R_0$ for $R_0 > 0$.
Then we have for every horosphere $\mathcal{H} \subset X$ and every $x,y \in \mathcal{H}$:
$$ d_{\mathcal{H}}(x,y) \le e^{d(x,y)\frac{\sqrt{R_0}}{2}} d(x,y). $$
\end{lemma}

\begin{proof}
Let $W^s(v) \subset SX$ be the unique stable manifold satisfying $\mathcal{H} = \pi(W^s(v))$.
For $x,y \in \mathcal{H}$ and $c: [0,1] \to X$ be the geodesic satisfying $c(0) = x$ and $c(1) = y$. Let $c_0 = P_{\mathcal{H}} \circ c: [0,1] \to \mathcal{H}$ be the orthogonal projection of $c$ into the horosphere $\mathcal{H}$ and $\gamma_0: [0,1] \to W^s(v)$ be the lift of $c_0$, that is, $c_0 = \pi \circ \gamma_0$. Then we can write $c(s) = c_{\gamma_0(s)}(r(s))$  with a smooth function $r: [0,1] \to \mathbb{R}$ satisfying $|r(s)| \le d(x,y)/2$ for all $s \in [0,1]$, and we have
\begin{align*}
d(x,y) &\ge \int_0^1 \Vert \dot c(s_0) \Vert ds_0 \\
&\ge \int_0^1 \left\Vert \frac{d}{ds}\vert_{s=s_0} c_{\gamma_0(s)}(r(s)) \right\Vert ds_0 \\
&= \int_0^1 \left( \Vert S_{\gamma_0(s_0)}(r(s_0)) (P^{\gamma_0(s_0)}_{r(s_0)} \dot c_0(s_0)) \Vert^2 + \Vert \dot c_{\gamma_0(s_0)}(r(s_0))\Vert ^2 \right)^{1/2} ds_0 \\
&\stackrel{(*)}{\ge} \int_0^1 \left( e^{-2 r(s_0) \sqrt{R_0}} \Vert P^{\gamma_0(s_0)}_{r(s_0)} \dot c_0(s_0) \Vert^2 
+ \Vert \dot c_{\gamma_0(s_0)}(r(s_0)) \Vert^2  
\right)^{1/2} ds_0 \\
&\ge e^{-d(x,y)\frac{\sqrt{R_0}}{2}} \int_0^1 \Vert \dot c_0(s_0) \Vert ds_0 \\
&\ge e^{-d(x,y)\frac{\sqrt{R_0}}{2}} d_{\mathcal{H}}(x,y).
\end{align*}
The above estimate $(*)$ is a consequence of the following inequality:
\begin{equation} \label{eq:S_vrest} 
\Vert S_v(r)(w) \Vert \ge e^{-r \sqrt{R_0}} \Vert w \Vert. 
\end{equation}
In the case $r \ge 0$, this inequality follows from (2.6) in Lemma I.2.16 in \cite{K0}. In the case $r < 0$, it suffices to prove \eqref{eq:S_vrest} for $\Vert w \Vert =1$. We consider
$$ f(r) = \Vert S_v(r) w \Vert^2 > 0. $$
Differentiation and Lemma I.2.17 in \cite{K0} yield
$$
f'(r) = 2 \langle \left( S_v'(r) S_v^{-1}(r) \right) S_v(r)w, S_v(r) w\rangle  \\
%&= 2 \langle S_{\phi^r v}'(0) S_v(r)w,S_v(r) w\rangle \\
\le 2 \sqrt{R_0} \langle S_v(r)w, S_v(r)w \rangle = 2 \sqrt{R_0} f(r).
$$
This implies
$$ \frac{d}{dt} \log f(t) \le 2 \sqrt{R_0}, $$
and therefore
$$ - \log f(r) = \int_r^0 \frac{d}{dt} \log f(t) dt \le 2 r \sqrt{R_0}, $$
and we obtain
$$ \Vert S_v(r) w \Vert = f^{1/2}(r) \ge e^{-r \sqrt{R_0}}, $$
finishing the proof of \eqref{eq:S_vrest} for the case $r < 0$.
\end{proof}

Now we are ready to prove continuity of $(p,v) \mapsto - \grad b_v(p)$ in $X \times SX_\alpha$ for $\alpha > 0$.
% We first prove, for fixed $p \in X$, continuity of the map $v \mapsto -\grad b_v(p)$ on $SX_\alpha$. 
We start with sequences $v_n \to v$ in $SX_\alpha$ with $p = \pi(v) = \lim p_n$, $p_n= \pi(v_n)$ and $q_n \to q$ in $X$. Let $w_n = - \grad b_{v_n}(q_n) \in SX_\alpha$ and $w = - \grad b_v(q) \in SX_\alpha$. We want to show that $w_n \to w$. Let $w_{n_j}$ be a subsequence converging to a vector $\tilde w \in SX_\alpha$ with $\pi(\tilde w) =q$. Let $t \ge 0$ be fixed. By convergence, we can find $j$ large enough such that
$$ d(c_{\tilde w}(t),c_{w_{n_j}}(t)), d(c_{v_{n_j}}(t),c_v(t)) \le 1  $$
and
$$ d(p_{n_j},q_{n_j}) < d(p,q)+1. $$
This implies that
\begin{align*}
d(c_{\tilde w}(t),c_v(t)) &\le d(c_{\tilde w}(t),c_{w_{n_j}}(t)) + d(c_{v_{n_j}}(t),c_{w_{n_j}}(t)) + d(c_{v_{n_j}}(t),c_v(t)) \\ 
&\le 1 + d(p_{n_j},q_{n_j}) + a_2 e^{-\frac{\rho}{2} t} d_{\mathcal{H}_{n_j}}(p_{n_j},P_{\mathcal{H}_{n_j}}(q_{n_j})) + 1 \\ &\le 2 + 
d(p_{n_j},q_{n_j}) + 2 a_2 e^{-\frac{\rho}{2} t} 
e^{2d(p_{n_j},q_{n_j}) \frac{\sqrt{R_0}}{2} } d(p_{n_j},q_{n_j}) \\
&\le 2 + (1 + 2a_2 
e^{(d(p,q)+1) \sqrt{R_0} }) (d(p,q)+1),
\end{align*}
with $\mathcal{H}_{n_j} = \pi(W^s(v_{n_j}))$, since 
$$ d(p_{n_j},P_{\mathcal{H}}(q_{n_j})) \le d(p_{n_j},q_{n_j}) + d(q_{n_j},P_{\mathcal{H}}(q_{n_j})) \le 2 d(p_{n_j},q_{n_j}). $$ 
The right hand side is independent of $t \ge 0$, and the geodesic rays $c_{\tilde w}$ and $c_v$ are therefore asymptotic.
Since $t \mapsto d(c_w(t),c_v(t))$ remains also bounded for all $t \ge 0$ and $w,\tilde w \in SX_\alpha$, we conclude that $\tilde w = w$, by Proposition \ref{prop:divgeod}. This implies $w_n \to w$, completing the continuity proof.

\begin{remark} The above arguments can be also used to prove continuity of the map $(p,v) \mapsto - \grad b_v(p)$ on $X \times SX$ for simply connected Riemannian manifolds $(X,g)$ without conjugate points, satisfying the following two properties:
\begin{itemize}
\item \emph{Divergence of geodesic rays}, that is, we have for every pair $v,w \in SX$ of different vectors with $p=\pi(v)=\pi(w)$,
$$ d(c_v(t),c_w(t)) \to \infty \quad \text{as $t \to \infty$.} $$
\item \emph{Asymptotic geodesics for stable manifolds}, that is, we have a function $f: [0,\infty) \to [0,\infty)$ such that, for every $v \in SX$ and $w \in W^s(v)$, 
$$ d(c_v(t),c_w(t)) \le f(d(\pi(v),\pi(w))) \quad \text{for all $t \ge 0$.} $$
\end{itemize}
\end{remark}

\subsection{Step 4: $\det D$ and $\trace D$ are both globally constant} \label{subsec:step4}

As in the previous subsection, we assume that $(X,g)$ is a rank one simply connected manifold with invariant horospherical mean curvature function, satisfying $\Vert R \Vert \le R_0$ and $\Vert \nabla R \Vert \le R_0'$. Let $\alpha > 0$ and $v_0 \in SX_\alpha$. Of central importance in this subsection is the map
$$ F=F_{v_0}: b_{v_0}^{-1}(0) \times X \to SX, \quad F(x,y) = - \grad b_{\grad b_{-v_0}(x)}(y). $$
Below, we show that $F$ is continuous and injective. For the injectivity,
the following result, which is similar to Proposition \ref{prop:divgeod}, is used.

\begin{proposition}[Divergence of geodesic rays on unstable leaves in $SX_\alpha$]
  \label{prop:divgeod2}
  Let $\alpha >0$, $v \in SX_\alpha$ and $w \in W^u(v)$ be two unit vectors with $\pi(v) \neq \pi(w)$. Then we have
  \begin{equation} \label{eq:dcvcw2} 
  d(c_v(t),c_w(t)) \to \infty \quad \text{as $t \to \infty$.} 
  \end{equation}
\end{proposition}

\begin{proof} Let $v,w \in W^u(v) \subset SX_\alpha$ be two unit vectors with $\pi(v) \neq \pi(w)$ as in the proposition. Let $t > 0$ be arbitrary and $\beta: [0,1] \to X$ be a curve with $\beta(0) = c_v(t)$ and $\beta(1) = c_w(t)$. If there exists $s \in (0,1)$ with $d(\beta(s),\pi W^u(v))\le t/2$, then $\ell(\beta) \ge t$. If $\beta$ satisfies $d(\beta([0,1]),\pi W^u(v)) \ge t/2$, then we can write
    $$ \beta(s) = c_{\gamma(s)}(r(s)) $$
    with $r: [0,1] \to [t/2,\infty)$ and $\gamma: [0,1] \to W^u(v)$, $\gamma(0)=v$, $\gamma(1)=w$. 
  We have
    $$ \ell(\beta) = \int_0^1 \Vert \dot \beta(s) \Vert ds \ge \int_0^1 \Vert U_{\gamma(s)}(r(s))(P_{r(s)} (\pi \circ \gamma)'(s)) \Vert ds \ge c e^{\frac{\rho t} {2}} \ell(\pi \circ \gamma) \ge c e^{\frac{\rho t}{2}} d(\pi(v),\pi(w)) $$
with a suitable constant $c > 0$ (see \cite[bottom of p. 110]{Bol-79}). Here $P_r$ denotes parallel transform along $c_{\gamma(s)}$. Since $\beta$ was an arbitrary curve connecting $c_v(t)$ and $c_w(t)$, this implies that we have \eqref{eq:dcvcw2}.
\end{proof}

We verify the following properties of $F$:
\begin{itemize}
\item[(i)] $F$ is continuous on $b_{v_0}^{-1}(0) \times X$: For $(x_n,y_n) \in b_{v_0}^{-1}(0)\times X$ with $(x_n,y_n) \to (x,y) \in b_{v_0}^{-1}(0) \times X$, we have $w_n = \grad b_{-v_0}(x_n) \to w = \grad b_{-v_0}(x)$ since $x_n \to x$ and $\grad b_{-v_0}$ is $C^1$. It follows from the previous subsection that
$$ F(x_n,y_n) = - \grad b_{w_n}(y_n) \to - \grad b_w(y) = F(x,y), $$
since $y_n \to y$. This proves continuity of $F$.
\item[(ii)] $F$ is injective: Assume $F(x_1,y_1) = F(x_2,y_2)$, that is,
$$ \grad b_{\grad b_{-v_0}(x_1)}(y_1) = \grad b_{\grad b_{-v_0}(x_2)}(y_2). $$
Since these are unit vectors with footpoints $y_1, y_2 \in X$, we conclude that $y_1=y_2$. Moreover, the geodesic rays with initial vectors $\grad b_{-v_0}(x_1), \grad b_{-v_0}(x_2) \in W^{0u}(v_0)$ are asymptotic. Since $x_1, x_2 \in b_{v_0}^{-1}(0)$, both vectors lie in $W^u(v_0)$, and Proposition \ref{prop:divgeod2} implies that $x_1 = x_2$. 
\end{itemize}

Since $(X,g)$ has rank one, there exists $\tilde v \in SX$ with $\delta = \det D(\tilde v) > 0$. We consider the non-empty closed subset
$$ W = \{ v \in SX: \det D(v) = \delta \} $$
of $SX$. This set is also open in $SX$ by the following argument: Let $v_0 \in W$. Then the above map $F_{v_0}: b_{v_0}^{-1}(0) \times X \to SX$ is injective and continuous and its image ${\rm{im}} F_{v_0}$ lies in $W$ by Subsections \ref{subsec:detconstgeod} and \ref{subsec:detconststable}. By Brouwer's Domain Invariance \cite{Brou}, the image ${\rm{im}} F_{v_0}$ contains an open neighbourhood of $v_0$ in $SX$. Since $W$ is a non-empty open and closed subset of the connected set $SX$, we have $W = SX$ and $\det D$ is constant on all of $SX$. In particular, all unit vectors in $SX$ are rank one. 

Let $\delta' = \trace D(\tilde v)$ and
$$ W' = \{ v \in SX: \trace D(v) = \delta' \}. $$
The same arguments, together with the fact that all unit vectors of $SX$ are rank one, show that $W' = SX$. This completes the steps in the proof of Theorem \ref{thm:main1}. \qed

\smallskip

We complete this subsection with a result that will be relevant in the next Section. 
%To state this result, we use a natural distance function $d_{SM}$ on the unit tangent bundle: $d_{SM}$ is 
% \begin{itemize}
%     \item either the Sasaki distance induced by the Sasaki metric $\langle \cdot,\cdot \rangle_{SM}$ on $SM$, defined as follows: for $X \in C^\infty((-\epsilon,\epsilon),SM)$,
% $$ \Vert \dot X(0) \Vert_{SM}^2 = \Vert \frac{d}{ds} (\pi \circ X)(0) \Vert^2 + \Vert \frac{D}{dt} X(0) \Vert^2, $$
% with $\frac{D}{dt}$ the covariant derivative along the footpoint curve $\pi \circ X: (-\epsilon,\epsilon) \to M$,
% \item or the dynamically defined metric for any fixed $t_0 > 0$, that is,
% $$ d_{SM}(v,v') = \max_{t \in [0,t_0]} d(c_v(t),c_{v'}(t)). $$
% \end{itemize}

\begin{thm} \label{thm:metatheorem} Let  $(X,g)$ be a simply connected rank one manifold with invariant horospherical mean curvature function and $\Vert R \Vert \le R_0$ and $\Vert \nabla R \Vert \le R_0'$ for some $R_0,R_0' > 0$. Let $f: SX \to \R$ be a continuous function which is reversible ($f(v)=f(-v)$ for all $v \in SX)$, invariant under the geodesic flow ($f(\phi^t(v)) = f(v)$ for all $v \in SX$ and all $t \in \R)$, and uniformly continuous on all stable manifolds in the following sense: for every $\epsilon >0$ there exists $\epsilon' >0$ such that, for all $v \in SX$ and $v' \in W^s(v)$ with $d_{\mathcal{H}}(\pi(v),\pi(v')) \le \epsilon'$, where $d_{\mathcal{H}}$ denotes the intrinsic distance in the horosphere $\mathcal{H} = \pi(W^s(v))$, we have
$$ |f(v')-f(v)| \le \epsilon. $$
Then $f$ is globally constant.
\end{thm}

\begin{proof}
    Let $v_0 \in SX$, $\delta_0 = f(v_0)$ and
    $$ W_0 = \{ v \in SX: f(v) = \delta_0 \}. $$
    Since $f$ is continuous, we know that $W_0$ is non-empty and closed in $SX$. It remains to show that $W_0$ is open to prove that we have $W_0 = SX$ and that $f$ is globally constant.
    
    To this end, we first show that $f$ is constant on all stable manifolds $W^s(v)$, that is, we have for all $v'\in W^s(v)$, 
    $$ |f(v')-f(v)| \le \epsilon $$
    for any $\epsilon >0$. We already know from the above discussion that $\det D: SX \to \R$ is constant with positive $\alpha = \det D$ and, therefore
    $$ D(v) \ge \rho \cdot \Id \quad \text{for all $v \in SX$,} $$
    with $\rho$ related to $\alpha$ as in \eqref{eq:alpharho}. Let $\gamma: [0,1] \to W^s(v)$ be a $C^1$-curve with $\gamma(0) = v$, $\gamma(1) = v'$, $\beta_t = \pi(\phi^t\gamma) \subset W^s(\phi^t(v)$ and $\beta = \beta_0$. Then Proposition \ref{prop:SLbeta-estimate} yields
    $$ d(\pi(\phi^t(v')),\pi(\phi^t(v)) \le \int_0^1 \Vert \dot \beta_t(s) \Vert ds \le a e^{-\frac{\rho}{2}t} \int_0^1 \Vert \dot \beta(s) \Vert ds, 
    $$
    and we can choose $t > 0$ large enough that 
    $$ d(\pi(\phi^t(v')),\pi(\phi^t(v)) \le \epsilon'. $$
    Then we obtain from the $\phi^t$-invariance of $f$ that
    $$ |f(v') - f(v)| = |f(\phi^t(v')) - f(\phi^t(v))| \le \epsilon, $$
    finishing this part of the proof.

    Since $f$ is reversible and constant on all stable manifolds, $f$ is also constant on all weakly stable and unstable manifolds, and we can apply Brouwer's Domain Invariance, as above, to conclude that the set $W_0$ is open in $SX$. This finishes the proof of the theorem.
\end{proof}

\subsection{$3$-dimensional manifolds with invariant horospherical mean curvature functions}

The classification of $3$-dimensional manifolds with invariant horospherical mean curvature functions is an interesting open problem. We have the following partial result.

\begin{proposition} \label{prop:3dimrankone}
    Let $(M,g)$ be a $3$-dimensional rank one manifold with invariant horospherical mean curvature function, satisfying $\Vert R \Vert \le R_0$ and $\Vert \nabla R \Vert \le R_0'$ for some constants $R_0,R_0'>0$. Then $M$ has constant curvature. 
    
    Moreover, if $M$ is compact, the curvature assumptions are automatically satisfied.
\end{proposition}

\begin{proof}
    Let $(M,g)$ be a manifold as in the proposition. Without loss of generality, we can assume that $M$ is simply connected. Theorem \ref{thm:main1} implies that $M$ is asymptotically harmonic with $h > 0$. The statement of the proposition follows now from \cite{SchSh-08}. 
\end{proof}

\section{Rank one D'Atri spaces without conjugate points}
\label{sec:rankoneDAtri}

This section is concerned with D'Atri spaces. We first prove our main Theorem \ref{thm:main2} from the introduction, and consider then the special case of homogeneous D'Atri spaces and $3$-dimensional D'Atri spaces without conjugate points.

\subsection{Proof of Theorem \ref{thm:main2}}

In this subsection, we prove Theorem \ref{thm:main2} from the introduction, namely, that rank one D'Atri spaces without conjugate points, continuous horospherical mean curvature functions and global Riemann curvature bounds are harmonic manifolds. Both types of manifolds are defined via orthogonal Jacobi tensors $A_v(t) \in {\rm{End}}(\phi^t(v)^{\bot})$ along geodesics $c_v$ with the initial conditions $A_v (0) = 0, A_v'(0) = \Id$.

% D'Atri and asymptotic D'Atri spaces provide a natural generalization of symmetric spaces and harmonic manifolds by requiring certain volume-preserving and curvature symmetry conditions.
% Let $A_v (t) \in End(c_v(t)^{\perp})$ be the Jacobi tensor along the geodesic $c_v$. That is, $A_v (t)$ satisfies the Jacobi equation with the initial conditions $A_v (0) = 0, A_v'(0) = id$. 
% Define the shape operator of a geodesic sphere in the direction $v$ at $t$ by $S_v(t) := A_v'(t) A_v^{-1}(t)$. 

\begin{definition}[{see \cite[Definition 4.2]{He2}}]
Let $(M,g)$ be a Riemannian manifold. 
\begin{itemize}
\item[(i)] $M$ is called \emph{D'Atri space} if
$
\det A_v(t) = \det A_{-v}(t), 
$
% where $A_v(t)$ is the orthogonal Jacobi tensor along the geodesic $c_v$ satisfying
% $
% A_v(0) = \mathrm{Id}, A_v(t) = 0.
% $
holds for all $t\geq 0$ and all unit vectors $v \in SM$. %\cite{Du}. 
\item[(ii)] $M$ is called a \emph{harmonic manifold} if there exists a function $f \in C^\infty([0,\infty))$ such that $\det A_v(t) = f(t)$ for all $t \ge 0$ and all unit vectors $v \in SM$.  
\end{itemize}
\end{definition}

% A Riemannian manifold $(M,g)$ is called a D'Atri space if $\det A_v(t) = \det A_{-v}(t)$  holds for all $v \in SM$ and for all $t \geq 0$, where $SM$ be the unit tangent bundle of $M$ \cite{Du}.
% \begin{proposition} \cite{Du} Let $(M, g)$ be a complete Riemannian manifold. If $(M, g)$ is a D’Atri space 
% then for every small $t > 0$, 
% the function $v \mapsto {\rm tr}(S_v(t))$ on $SM$ is invariant under geodesic flow. 
%  \end{proposition} 
%  \begin{proof} The geodesic flow is is defined by
% $g_s(v)= (c_v (s), c'_v(s))$ for any $v\in SM$ and $s \in \mathbb{R}$. Since $(M, g)$ is a D’Atri space, we have $\det A_v(t)= \det {A_{g_s (v)}}(t)$ for all $t> 0$ and $s\ge 0$ \cite{He2}. By linear algebra, we obtain
% \begin{eqnarray*}\displaystyle{{\rm tr} S_v (t) = \frac{(\det A_v(t))'}{\det A_v (t)} = \frac{(\det A_{g_s (v)}(t))'}{\det A_{- {g_s (v)}} (t)} = {\rm tr}\, {S}_{g_s (v)}(t) }.\end{eqnarray*}
% Therefore, the function $h(v) = {\rm tr} (S_v(t))$ is constant along the geodesic flow.
%  \end{proof}

Geometrically, D'Atri spaces are those in which the geodesic inversions at all points are volume preserving. Harmonic manifolds can be described as those spaces in which geodesic spheres of the same radius have constant mean curvature. The function $f \in C^\infty([0,\infty))$ in a harmonic manifold agrees -- up to a constant multiplicative factor -- with the volume density of spheres. Note that D'Atri spaces are real analytic and that $v \mapsto \det A_v(t)$ is invariant under the geodesic flow.

\begin{proposition}[see {\cite[Lemma 4.6]{He2}}] \label{prop:datrigeodaetinv}
Let $(M,g)$ be a D'Atri space. Then $M$ is real analytic and the map $v \mapsto A_v(t)$ is invariant under the geodesic flow for all $t > 0$. 
\end{proposition}

We have the following result about the horospherical mean curvature function for D'Atri spaces without conjugate points. 

% In particular, for non-compact D’Atri spaces without conjugate points, the invariance under the geodesic flow leads to the following result.
 % \begin{proposition}
 % Let $(M,g)$ be non-compact D'Atri spaces without conjugate points have the property that the horospherical mean curvature function $h: SM \rightarrow [0,\infty)$ given by $h(v) = - {\rm tr} (\mathcal S_v'(0))$ is invariant under the geodesic flow and satisfies h(v)=h(-v).
 %  \end{proposition}
 % \begin{proof}
 % \end{proof}  

\begin{proposition} \label {prop:DAhoromean}
Let $(M, g)$ be a D'Atri Space without conjugate points. Then the horospherical mean curvature function $h$ is reversible and invariant under the geodesic flow.  
% and $h(v) = h(-v)$ and $h(v)\geq 0$ for all  $v \in SM$.
\end{proposition}

\begin{proof} 
Recall from Proposition \ref{prop:datrigeodaetinv} that the function $f_t(v) = \det A_v(t)$ is invariant under the geodesic flow. 
% For all $t \ge 0$, $f_t: SM \to \R$ is invariant under the geodesic flow, that is, $f_t(\phi^s(v)) = f_t(v)$ for all $s \in \R$ (see \cite[Lemma 4.6]{He2}).
% In a D'Atri space, $f_t(v) = \det A_v(t)$ is invariant under the geodesic flow for all 
% $t \ge 0$ \cite[Lemma 4.6]{He2}. 
Since the orthogonal Jacobi tensor $Y(s)=A_v(s+t)\left(A_v^{-1}(t)\right)_{s+t}$ along $c_v$ (with $\left( A_v^{-1}(t)\right)_{s+t}$ the parallel transport of $A_v^{-1}(t)\in {\rm{End}}(\phi^t(v)^\bot)$ along $c_v$ to the point $c_v(s+t)$) satisfies $Y(-t)=0$ and $Y(0)=\Id$, we have $Y(s) = U_{\phi^t(v),t}(s)$, and therefore 
$$ U_{v,t}'(0) = A'_{\phi^{-t}(v)}(t) A_{\phi^{-t}(v)}^{-1}(t). $$
% Consider a point $p$ and a tangent vector $v$ at $p$.
% Then
% $$
% U_{v,t}'(0) = A_{\phi^{-t} v}'(t) A_{\phi^{-t} v}^{-1}(t)
% $$
Taking the trace on both sides, we obtain
$$
\operatorname{tr}  U_{v,t}'(0) = \operatorname{tr} (A_{\phi^{-t}(v)}'(t) A_{\phi^{-t}(v)}^{-1}(t)).
$$
Fix $s \in \mathbb{R}$: Differentiating the flow-invariant $\det A_{\phi^s v}(t)=\det A_v(t)$ with respect to $t$ yields
\begin{align}
\frac{d}{dt} \det A_{\phi^s(v)}(t) &= \operatorname{tr} \left( A_{\phi^s(v)}'(t) A_{\phi^s(v)}^{-1}(t) \right) \cdot \det A_{\phi^s(v)}(t) \notag\\
&= \frac{d}{dt} \det A_v(t) = \operatorname{tr} \left( A_v'(t) A_v^{-1}(t) \right) \cdot \det A_v(t).
\end{align}
Due to the absence of conjugate points, we have $\det A_v(t) = \det A_{\phi^s(v)}(t) \neq 0$ for $t > 0$), it follows that
$$
\mathrm{tr} \left( A_{\phi^s(v)}'(t) A_{\phi^s(v)}^{-1}(t) \right) = \mathrm{tr} \left( A_v'(t) A_v^{-1}(t) \right)
$$
for all $s \in \mathbb{R}$.
%%%
Consequently:
\begin{align}
\operatorname{tr} \,{U} ^{\prime}_{\phi^s(v),t} (0)
&= \operatorname{tr} \left( A'_{\phi^{-t+s}(v)}(t)  A^{-1}_{\phi^{-t+s}(v)}(t) \right) \notag \\
&= \operatorname{tr} \left( A'_{\phi^{-t}(v)}(t)  A_{\phi^{-t}(v)}^{-1}(t) \right) = \operatorname{tr} \, U^{\prime}_{v,t}(0).
\end{align}
Taking the limit, as $t \rightarrow \infty$, we obtain for all $s \in \R$:
$$ h(\phi^s(v)) = \lim_{t \to \infty} \operatorname{tr}\left(U^{\prime}_{\phi^s(v),t}(0) \right)
= \lim_{t \to \infty} \operatorname{tr}\left(U^{\prime}_{v,t}(0) \right) = h(v).
$$
This shows that $h$ is invariant under the geodesic flow.

Next, since
$
\det A_v(t) = \det A_{-v}(t), 
$ it follows by similar arguments that
\begin{equation} \label{eq:flipA}
 \operatorname{tr}(A_v'(t) A_v^{-1}(t)) =  \operatorname{tr}(A_{-v}'(t) A_{-v}^{-1}(t)).
 \end{equation}
Therefore, we have
$$\operatorname{tr}\left( {U}_{v,t}^{\prime}(0) \right) =  \operatorname{tr}\left( A_{\phi_{-t} v}'(t) A_{\phi_{-t} v}^{-1}(t) \right),
$$
and similarly:
$$
\operatorname{tr}({U}_{-v,t}^{\prime}(0)) = \operatorname{tr}(A_{\phi_{-t}(-v)}'(t) A_{\phi_{-t}(-v)}^{-1}(t)).
$$
Applying \eqref{eq:flipA} twice, we obtain:
% $$
%  \operatorname{tr}(A_{-\phi^{-t}(-v)}'(t) A_{-\phi^{-t}(-v)}^{-1}(t)) = \operatorname{tr}(A_{\phi^t v}'(t) A_{\phi^t v}^{-1}(t)) =
%   \operatorname{tr}({U}_{\phi^{2t} v.t}^{\prime}(0)) = \operatorname{tr}({U}_{v,t}^{\prime}(0)).  $$
$$  \operatorname{tr}({U}_{-v,t}^{\prime}(0)) = \operatorname{tr}({U}_{v,t}^{\prime}(0)).  $$
Taking again the limit, as $t\rightarrow \infty$, yields
$$h(v) = \lim_{t \to \infty} \operatorname{tr}({U}_{v,t}^{\prime}(0)) = \lim_{t \to \infty} \operatorname{tr}({U}_{-v,t}^{\prime}(0)) = h(-v).$$
This establishes the reversibility $h(v) = h(-v)$ for all $v \in SM$.
% Finally, from the properties of the Jacobi tensor:
% $$
% {U}_{v,t}(s) = S_{-v,t}(-s),$$
% which implies $${U}_{v,t}^{\prime}(0) = -S_{-v,t}^{\prime}(0).
% $$
% Taking traces:
% $$
%  \operatorname{tr}(S_{-v}^{\prime}(0)) = - \operatorname{tr}({U}_{v}^{\prime}(0)) = -h(v),
% $$
% abd similarly, 
% $$
%  \operatorname{tr}(S_{v}^{\prime}(0)) = -h(-v).
% $$
% Combining these results:
% $$
% 0 \leq \operatorname{tr}({U}_v^{\prime}(0) - S_v^{\prime}(0)) = h(v) + h(-v) = 2h(v).$$
% Thus, $h(v) \geq 0$. This completes the proof. 
\end{proof}

\begin{proposition} \label{prop:detdetestimate}
Let $(X,g)$ be a simply connected Riemannian manifold without conjugate points. Assume there exist $R_0, R_0' >0$ such that $\Vert R \Vert \le R_0$ and $\Vert \nabla R \Vert \le R_0'$ and $\rho > 0$ such that
$$ D(v) \ge \rho \cdot \Id \quad \text{for all $v \in SX$.} $$
Then there exists, for every fixed $t > 0$, a constant $C_t > 0$ such that we have for any two vectors $v, v' \in W^s(v)$ 
$$ | \det A_{v'}(t) - \det A_v(t) | \le C_t d_{\mathcal{H}}(\pi(v'),\pi(v)), $$
where $d_{\mathcal{H}}$ is the intrinsic distance in the horosphere $\mathcal{H} = \pi(W^s(v))$.
% % and any $C^1$-curve $\gamma: [0,1] \to W^s(v)$ with $\gamma(0)=v$ and $\gamma(1)=v'$,
% $$ | \det A_{v'}(t) - \det A_v(t) | \le C_t \ell(\beta) $$
% % with $\beta = \pi \circ \gamma$.
\end{proposition}

\begin{proof} Let $v,v' \in W^s(v)$, $\beta: [0,1] \to \mathcal{H}$ be a $C^1$-curve satisfying $\beta(0) =\pi(v)$ and $\beta(1) = \pi(v')$. Let $\gamma: [0,1] \to W^s(v)$ be the lift of $\beta$. Let $e_1,\dots,e_{n-1}: [0,1] \to S\mathcal{H}$ (with $n = \dim(X)$) be an orthonormal frame in $\mathcal{H} = \pi(W^s(v))$ along $\beta$, which is parallel within the horosphere $\mathcal{H}$ with respect to the induced metric as a hypersurface. Let $E_1(s,t),\dots,E_{n-1}(s,t)$ be the parallel transports of $e_1(s),\dots,e_{n-1}(s)$ along the geodesic $c_{\gamma(s)}(t)$ in $X$.      

For fixed $t > 0$, we define $f_t: SX \to \R$ by $f_t(v) = \det A_v(t)$. In the following arguments, all our orthogonal $(1,1)$-tensors $T(s,t)$ are expressed as matrices with respect the basis $E_1(s,t),\dots,E_{n-1}(s,t)$, that is, the matrix entries are given by $\langle T(s,t) E_i(s,t),E_j(s,t) \rangle$. Then we have 
\begin{align} 
|f_t(v') - f_t(v)| &\le \int_0^1 \left| \frac{\partial}{\partial s} f_t(\gamma(s)) \right| ds \nonumber \\
&= \int_0^1 \left| \trace\left( \frac{\partial}{\partial s} A_{\gamma(s)}(t) A_{\gamma(s)}^{-1}(t) \right) \det\left(A_{\gamma(s)}(t)\right) \right| ds. \label{eq:ftvftv'}
\end{align}
Next, we provide estimates for the norms of all $(1,1)$-tensors appearing in the integral on the right hand side of \eqref{eq:ftvftv'}.
Note that we have
\begin{equation} \label{eq:C1AC2}
C_1 \Vert x \Vert \le \Vert A_w(t)x \Vert \le C_2 \Vert x \Vert \quad \text{for all $w \in SM$ and $x \in \phi^t(w)^\bot$,}
\end{equation}
with constants $C_1, C_2 > 0$ only depending on $R_0,t$. The upper bound in \eqref{eq:C1AC2} is a direct consequence of \cite[Proposition I.2.11(b)]{K0} providing $C_2 = \frac{1}{\sqrt{R_0}} \sinh(\sqrt{R_0}t)$. The lower bound in \eqref{eq:C1AC2} is a direct consequence of \cite[Lemma I.2.13]{K0}. (Lemma I.2.13 in \cite{K0} states this result for any $t \ge t_0=1$, but standard rescaling arguments imply that the result holds for any $t_0 > 0$.) This implies $|\det\left(A_{\gamma(s)}(t)\right)| \le C_2^{n-1}$ with $n = \dim(M)$ and $\Vert A_{\gamma(s)}^{-1}(t) \Vert \le 1/C_1$. It remains to estimate $\Vert \frac{\partial}{\partial s} A_{\gamma(s)}(t) \Vert$. Let
$$ Z_{\gamma(s)}(t) = \frac{\partial}{\partial s} A_{\gamma(s)}(t). $$
For fixed $s \in [0,1]$, $Z_{\gamma(s)}(t)$ satisfies the following matrix valued ordinary differential equation 
\begin{align*} Z''_{\gamma(s)}(t) &= \frac{\partial}{\partial s} \frac{\partial^2}{\partial t^2} A_{\gamma(s)}(t) = - \frac{\partial}{\partial s} \left( R_{\gamma(s)}(t) A_{\gamma(s)}(t) \right)\\ 
&= - \left( \frac{\partial}{\partial s} R_{\gamma(s)}(t) \right) A_{\gamma(s)}(t) - R_{\gamma(s)}(t) Z_{\gamma(s)}(t) .%\left( \frac{\partial}{\partial s} A_{\gamma(s)}(t) \right) 
\end{align*}
Rewriting this ordinary differential equation leads to
$$ \begin{pmatrix} Z_{\gamma(s)} \\ Z'_{\gamma(s)} \end{pmatrix}'(t) = \begin{pmatrix} 0 & \Id \\ -R_{\gamma(s)}(t) & 0 \end{pmatrix} \begin{pmatrix} Z_{\gamma(s)} \\ Z'_{\gamma(s)}\end{pmatrix}(t) + \begin{pmatrix} 0 \\ - \left( \frac{\partial}{\partial s} R_{\gamma(s)}(t) \right) A_{\gamma(s)}(t)\end{pmatrix}, $$
with the initial conditions
$$ \begin{pmatrix} Z_{\gamma(s)}(0) \\ Z'_{\gamma(s)}(0) \end{pmatrix} = \begin{pmatrix} 0 \\ 0 \end{pmatrix}, $$
since $A_{\gamma(s)}(0) = 0$ and $A'_{\gamma(s)}=\Id$. That is, our ordinary differential equation is of the form
$$ x'(t) = F(t,x(t)) \quad \text{with $x(0) = 0$} $$ 
and
\begin{align*}
x(t) &= \begin{pmatrix} Z_{\gamma(s)}(t) \\ Z'_{\gamma(s)}(t) \end{pmatrix} \in  {\rm{End}}(\R^{n-1}) \times {\rm{End}}(\R^{n-1}), \\
F(t,x) &= a_s(t) x + q_s(t), \\
a_s(t) &= \begin{pmatrix} 0 & \Id \\ -R_{\gamma(s)}(t) & 0 \end{pmatrix}, \\
q_s(t) &= \begin{pmatrix} 0 \\ - \left( \frac{\partial}{\partial s} R_{\gamma(s)}(t) \right) A_{\gamma(s)}(t)\end{pmatrix}.
\end{align*}
It follows from the proof of \cite[Proposition 7.8]{Am} that
$$ \Vert x(t) \Vert \le \int_0^t \Vert q_s(\tau) \Vert d\tau + \int_0^t \Vert a_s(\tau) \Vert \cdot \Vert x(\tau) \Vert d\tau. $$
Applying \cite[Corollary 6.2]{Am}, we obtain the estimate
$$ \Vert\frac{\partial}{\partial s} A_{\gamma(s)}(t) \Vert \le \Vert x(t) \Vert \le \left( \int_0^t \Vert q_s(\tau) \Vert d\tau \right) e^{\int_0^t \Vert a_s(\tau) \Vert d\tau}. $$
For $0 \le \tau \le t$, it follows from \eqref{eq:C1AC2} that $\Vert a_s(\tau) \Vert \le 1 + R_0$ and from \cite[(2.14)]{KP} that 
$$ \Vert q_s(\tau) \Vert \le C_2 \, C \, \Vert \dot \beta(s) \Vert, $$
with a constant $C$ only depending on $R_0,R_0',\rho,t$. Consequently, we end up with the estimate
$$ \Vert\frac{\partial}{\partial s} A_{\gamma(s)}(t) \Vert \le t \, C_2 \, C \, e^{t (1+R_0)} \, \Vert \dot \beta(s) \Vert = C' \, \Vert \dot \beta(s) \Vert, $$
where $C'$ depends only on $R_0,R_0',\rho,t$. Plugging these results into \eqref{eq:ftvftv'}, we obtain
\begin{align}
| \det A_{v'}(t) - \det A_v(t) | &\le (n-1) \int_0^1 \Vert \frac{\partial}{\partial s} A_{\gamma(s)}(t)\Vert \cdot \Vert A^{-1}_{\gamma(s)}(t) \Vert \cdot | \det A_{\gamma(s)}(t) | ds \nonumber \\
&\le (n-1) \, \frac{C'\, C_2^{n-1}}{C_1} \, \ell(\beta),
\label{eq:detdiff}
\end{align}
with constants only depending on $R_0,R_0',\rho,t$. Since $\beta: [0,1] \to \mathcal{H}$ was an arbitrary $C^1$-curve connecting $\pi(v)$ and $\pi(v')$, $\ell(\beta)$ in \eqref{eq:detdiff} can be replaced  by $d_{\mathcal{H}}(\pi(v),\pi(v'))$.
\end{proof} 

Now we can present the proof of our second main result in the introduction. 

\begin{proof}[Proof of Theorem \ref{thm:main2}]
Without loss of generality, it suffices to prove the theorem for any simply connected rank one D'Atri space $(X,g)$ without conjugate points, continuous horospherical mean curvature functions $h: SX \to \R$, and satisfying $\Vert R \Vert \le R_0$ and $\Vert \nabla R \Vert \le R_0'$. Let $t > 0$ and $f_t(v) = \det A_v(t)$. Since $X$ is a D'Atri space, $f_t$ is reversible and invariant under the geodesic flow, by Proposition \ref{prop:datrigeodaetinv}.
We have to prove that $f_t$ is constant.

It follows from Proposition \ref{prop:DAhoromean} that $X$ is a manifold with invariant horospherical mean curvature function $h$. Step 4 for such manifolds of rank one (see Subsection \ref{subsec:step4}) shows that $\det D(v) = \alpha > 0$ for all $v \in SX$, and therefore $D(v) \ge \rho \Id$ with $\rho$ given in \eqref{eq:alpharho}. This guarantees that we can apply Proposition \ref{prop:detdetestimate} and obtain, for all pairs $v,v' \in W^S(v)$ and $\mathcal{H} = \pi(W^s(v))$.
$$ | f_t(v') - f_t(v) | \le C_t d_{\mathcal{H}}(\pi(v'),\pi(v)). $$
Then all conditions of Theorem \ref{thm:metatheorem} are satisfied, in particular the uniform continuity with $\epsilon' = \frac{\epsilon}{C_t}$, and therefore, $f_t(v) = \det A_v(t)$ is constant, finishing the proof.
\end{proof}

\subsection{Homogeneous D'Atri spaces}

In \cite[Theorem 1.1]{He}, Heber showed that all non-compact, simply connected homogeneous harmonic manifolds are either flat, rank one symmetric or non-symmetric Damek-Ricci spaces.
This result, together with his earlier result in \cite[Theorem 4.7]{He2}, leads to the following characterization of non-positively curved homogeneous D'Atri spaces.

\begin{thm}[{see \cite[Theorem 4.7]{He2} and \cite[Theorem 1.1]{He}}] Let $(X,g)$ be a homogeneous, irreducible, simply connected D'Atri space of non-positive curvature. Then $X$ is either Euclidean, a symmetric space of non-compact type, or a non-symmetric Damek-Ricci space.
\end{thm}

Druetta \cite{Du} considered a special class of homogeneous D'Atri spaces, namely those of Iwasawa type. They are solvable Lie groups $S$ whose underlying Lie algebra $\mathfrak{s}$ has an orthogonal decomposition $\mathfrak{s} = \mathfrak{n} \oplus \mathfrak{a}$ with $\mathfrak{n} = [\mathfrak{s},\mathfrak{s}]$ and abelian $\mathfrak{a}$, satisfying some additional algebraic conditions (see \cite[Section 2]{Du}). The \emph{algebraic rank} $\rank_{\rm{alg}}(S)$ of such a manifold is the dimension of $\mathfrak{a}$. Druetta proved the following result:

\begin{thm}[{see \cite[Corollary 2.3]{Du} and \cite[Theorem 2.5]{Du}}] A homogeneous D'Atri space of Iwasawa type has no conjugate points. If such a space has algebraic rank one, it is a harmonic manifold and, therefore, a rank one symmetric space or a non-symmetric Damek-Ricci space.
\end{thm}

Computations similar to the ones carried out in \cite[Section 1]{Du0} yield the following:

\begin{lemma} Let $S$ be an $n$-dimensional rank one homogeneous space of Iwawasa type with $\mathfrak{a} = \R H_0$, $\Vert H_0 \Vert = 1$, and $\lambda_1,\dots,\lambda_n  > 0$ be the eigenvalues of ${\rm{ad}}_{H_0}\vert_{\mathfrak{n}}$ with corresponding eigenvectors $X_1,\dots,X_{n-1} \in \mathfrak{n}$. Let $X_1(t),\dots,X_{n-1}(t)$ be their parallel extensions along the geodesic $c(t) = \exp(tH_0)$. Then we have
$$ U_{\dot c(s)}(t)(X_i(s)) = e^{t\lambda_i} X_i(s+t) $$
and
$$ S_{\dot c(s)}(t)(X_i(s)) = e^{-t\lambda_i} X_i(s+t). $$
\end{lemma}

This lemma has the following consequence.

\begin{corollary} An rank one homogeneous space $S$ of Iwasawa type has also geometric rank one.
\end{corollary}

\begin{proof} The $(1,1)$-tensors $U(H_0)$ and $S(H_0)$ are given by
$$ U(H_0) = \begin{pmatrix} \lambda_1 & 0 & \dots & 0 \\ 0 &\lambda_2 & \dots & 0 \\ \vdots & & & \vdots  \\ 0 & 0 & \dots & \lambda_n \end{pmatrix} \quad \text{and} \quad S(H_0) = \begin{pmatrix} -\lambda_1 & 0 & \dots & 0 \\ 0 &-\lambda_2 & \dots & 0 \\ \vdots & & & \vdots  \\ 0 & 0 & \dots & -\lambda_n \end{pmatrix}, $$
with respect to the basis $X_1,\dots,X_{n-1} \in H_0^\bot$. Since all eigenvalues $\lambda_j$ are positive, $D(H_0) = U(H_0) - S(H_0)$ has trivial kernel. Therefore, $c(t)=\exp(tH_0)$ is a rank one geodesic and $\rank(S) = 1$.
\end{proof}

Our main Theorem \ref{thm:main2} has the following consequence for homogeneous D'Atri spaces.

\begin{corollary} Let $(X,g)$ be a homogeneous simply connected D'Atri space without conjugate points and continuous horospherical mean curvature function. If such a space $X$ has rank one, it is a rank one symmetric space of noncompact type or a non-symmetric Damek-Ricci space. 
\end{corollary}

\subsection{$3$-dimensional D'Atri spaces without conjugate points}

The classification of all $3$-dimensional D'Atri spaces goes back to Kowalski \cite{Ko}. This result implies the following classification for $3$-dimensional D'Atri spaces without conjugate points.

\begin{proposition} \label{prop:3dimDAtri}
Let $(X,g)$ be a simply connected $3$-dimensional D'Atri space without conjugate points. Then $X$ is -- up to scaling -- isometric to the Euclidean space $\R^3$, the Riemannian product $\Hh^2 \times \R$, or $\Hh^3$, where $\Hh^n$ is the $n$-dimensional hyperbolic space of constant curvature $-1$.
\end{proposition}

\begin{proof} We start with the following classification of all $3$-dimensional simply connected D'Atri spaces, given in \cite[Theorem 2]{Ko}:
\begin{enumerate}
    \item[(a)] For $c>0$,
    $$\R^3, \quad \Ss^3(c), \quad \Hh^3(-c), \quad \Ss^2(c) \times \R, \quad \Hh^2(-c)\times \R, $$ where $\Ss^n(c) \subset \R^{n+1}$ denotes the round sphere of curvature $c$,
\item[(b)]
$SU(2)$ (which is diffeomorphic to the $3$-sphere, but with a more general left-invariant metric, see \cite{Ko} for more details),
\item[(c)]
The universal covering $X = \widetilde{SL(2,\R)}$ of $SL(2,\R)$ with a left-invariant metric. $X$ has the following global coordinate system $\R^3 \ni (t,x,y) \to \varphi(t,x,y) \in X$, and the metric is explicitely given by
$$ds^2=\frac{1}{ \left | a+b\right |}dt^2+ \left | a+b\right |e^{-2t}dx^2+(dy+\sqrt{2b}e^{-t}dx)^2,$$ 
with parameters $b>0$ and $a+b<0$.
\item[(d)]
The Heisenberg group $X$ with any left invariant metric. $X$ has the following global coordinate system $\R^3 \ni (x,y,z) \to \varphi(x,y,z) \in X$, and the metric is explicitely given by
$$ds^2=\frac{1}{b}dx^2+ dz^2+(dy -xdz)^2,$$ with a parameter $b>0$.
\end{enumerate}
Since compact simply connected Riemannian manifolds have always conjugate points, the only surviving examples in (a) are $\R^3, \Hh^3(-c)$ and $\Hh^2(-c) \times \R$. For the same reason, $SU(2)$ with any left-invariant metric has conjugate points. Moreover, every non-abelian nilpotent Lie group with left invariant metric has conjugate points (see \cite[Corollary 2]{O'Sul}), which eliminates the Heisenberg group in (d).
To complete the classification result, we need to show that the manifolds in (c) have conjugate points. 

\medskip

Let $X = \widetilde{SL(2,\R)}$  with global coordinate system $\R^3 \ni (t,x,y) \to \varphi(t,x,y) \in X$.
The covariant derivatives are given by
\begin{align*}
\nabla_{\partial_t} \partial_t &= 0, \\
\nabla_{\partial_t} \partial_x = \nabla_{\partial_x} \partial_t &= \left( -1-\frac{b}{|a+b|}\right)\partial_x + \frac{\sqrt{b}(|a+b|+2b)e^{-t}}{\sqrt{2}|a+b|} \partial_y,\\
\nabla_{\partial_t} \partial_y = \nabla_{\partial y} \partial_t &= - \frac{\sqrt{b}e^t}{\sqrt{2}|a+b|}\partial_x + \frac{b}{|a+b|}\partial_y, \\
\nabla_{\partial_x} \partial_y = \nabla_{\partial y}\partial x &= \frac{|a+b|\sqrt{b}e^{-t}}{\sqrt{2}} \partial_t, \\
\nabla_{\partial_y} \partial_y &= 0,
\end{align*}
and the curve $\gamma(s) = \varphi(0,t,s)$ with $\dot \gamma(s) = \partial_y$ is a unit speed geodesic for any fixed choice of $t \in \R$.
We show that there are conjugate points along $\gamma$. We have the following results for the Riemannian curvature tensor:
\begin{align*}
R(\partial_t,\partial_y)\partial_y&=\frac{b}{2}\partial_t, \\
R(\partial_x,\partial_y)\partial_y&=-\frac{b}{2}\partial_x+\frac{b\sqrt{2b}e^{-t}}{2}\partial_y.
\end{align*}
Note that the global vector fields
$$ V_1 = \partial_t, \quad V_2 = \partial_x - \sqrt{2b} e^{-t} \partial_y, \quad V_3 = \partial_y $$
are pairwise orthogonal with
\begin{align*}
\nabla_{V_3}V_1=&-\frac{\sqrt{2b}e^t}{2|a+b|}\left(\partial_x-\sqrt{2b}e^{-t}\partial_y\right)=\alpha V_2,\\
\nabla_{V_3}V_2=&\nabla_{\partial_y}\partial_x-\sqrt{2b}e^{-t}\nabla_{\partial_y}\partial_y
=\frac{1}{2}\sqrt{2b}|a+b|e^{-t}V_1
=\beta V_1,
\end{align*}
where 
\begin{equation*}
\alpha=-\frac{\sqrt{2b}e^t}{2|a+b|}, \quad \beta=\frac{1}{2}\sqrt{2b}|a+b|e^{-t},
\end{equation*}
and
\begin{align*}
R(V_1,V_3)V_3&=\frac{b}{2}V_1,\\
R(V_2,V_3)V_3&=-\frac{b}{2}\left(\partial_x-\sqrt{2b}e^{-t}\partial_y\right)=-\frac{b}{2}V_2.
\end{align*}
We introduce the following vector field along $\gamma$:
$$ J(s) = u_1(s) V_1(\gamma(s)) + u_2(s) V_2(\gamma(s)), $$
and we obtain
\begin{align*}    J'(s)&=u_1'(s)V_1+u_2(s)\nabla_{\partial_y}V_2+u_2'(s)V_2+u_1(s)\nabla_{\partial_y}V_1\\
    &=(u_1'(s)+\beta u_2(s))V_1+(u_2'(s)+\alpha u_1(s))V_2,\\
    J^{\prime\prime}(s)&=(u_1^{\prime\prime}(s)+2\beta u_2'(s)+\alpha\beta u_1(s))V_1
    +(u_2^{\prime\prime}(s)+2\alpha u_1'(s)+\alpha\beta u_2(s))V_2,\\
R(J(s),V_3)V_3&=\frac{b}{2}u_1(s)V_1-\frac{b}{2}u_2(s)V_2.
\end{align*}
It follows that $J(s)$ satisfies the Jacobi equation 
$J^{\prime\prime}(s)+R(J(s),\partial_y)\partial_y=0$ if and only if
\begin{align*}
u_1''+\sqrt{2b}|a+b|e^{-t}u_2' &= 0, \\
u_2'' - \frac{\sqrt{2b}}{ |a+b|}e^t u_1' - b u_2 &= 0.
\end{align*}
It is easy to see that
\begin{align*}
u_1(s) &= \sqrt{2}|a+b|e^{-t}\left(\cos(\sqrt{b} s)-1\right), \\
u_2(s) &= \sin(\sqrt{b} s)
\end{align*}
are solutions of this system of differential equations and that, for this choice of $u_1, u_2$, the Jacobi field $J(s)$ satisfies $J(0)=0$ and $J(2\pi/\sqrt{b})=0$. 
This complete the proof that the manifolds in (c) have conjugate points.
\end{proof}

In view of the classification results Propositions \ref{prop:3dimrankone} and \ref{prop:3dimDAtri}, is is natural to conjecture the following.

\begin{conjecture} \label{conj:3dimclass}
Let $(X,g)$ be a $3$-dimensional simply connected manifold with invariant horospherical mean curvature function. Then $X$ is --  up to scaling -- isometric to the Euclidean space $\R^3$, the Riemannian product $\Hh^2 \times \R$, or $\Hh^3$.
\end{conjecture}
{\it {Acknowledgments}}.
The second author thanks the Department of Mathematics, Durham University, for the hospitality during the visit 
when this work was carried out. The research of JP was supported by the LMS and the National Research Foundation 
of Korea (NRF) grant funded by the Korea government (MSIT) (RS-2024-00334956), whose support made this research possible.

\end{document}